\documentclass{amsart}
\usepackage{amsmath,amssymb,latexsym, amsfonts, amscd, amsthm}

\usepackage{bbm}
\usepackage{enumerate}
\usepackage[all]{xy}

\allowdisplaybreaks
\newtheorem{thm}{Theorem}[section]
\newtheorem{pro}[thm]{Proposition}
\newtheorem{lm}[thm]{Lemma}
\newtheorem{cor}[thm]{Corollary}
\newtheorem{conj}[thm]{Conjecture}
\numberwithin{equation}{section}

\theoremstyle{remark}
\newtheorem{exa}[thm]{Example}
\newtheorem{rem}[thm]{Remark}

\theoremstyle{definition}
\newtheorem{defn}[thm]{Definition}

\DeclareMathOperator*{\End}{End}
\DeclareMathOperator*{\reg}{reg}

\DeclareMathOperator*{\Irr}{Irr}
\DeclareMathOperator*{\disc}{disc}
\DeclareMathOperator*{\Gal}{Gal}

\DeclareMathOperator*{\ad}{ad}

\DeclareMathOperator*{\Nrd}{Nrd}
\DeclareMathOperator*{\der}{der}
\DeclareMathOperator*{\diag}{diag}
\DeclareMathOperator*{\Stab}{Stab}

\DeclareMathOperator*{\Hom}{Hom}

\newcommand{\s}{\simeq}

\newcommand{\tsigma}{\widetilde{\sigma}}
\newcommand{\esq}{\mathcal{E}^{2}}

\newcommand{\CC}{\mathbb{C}}

\newcommand{\QQ}{\mathbb{Q}}

\newcommand{\ii}{\mathbf{\textit{i}}}

\def\bA{\bold A}
\def\bG{\bold G}
\def\bP{\bold P}
\def\bM{\bold M}
\def\bN{\bold N}

\newcommand{\tG}{\widetilde{G}}
\newcommand{\tbG}{\widetilde{\bold G}}
\newcommand{\tP}{\widetilde{P}}
\newcommand{\tbP}{\widetilde{\bold P}}
\newcommand{\tM}{\widetilde{M}}
\newcommand{\tbM}{\widetilde{\bold M}}
\newcommand{\tN}{\widetilde{N}}
\newcommand{\tbN}{\widetilde{\bold N}}

\newcommand{\tbA}{\widetilde{\bold A}}

\title[Transfer of $R$-groups between $p$-adic inner forms of $SL_n$]{Transfer of $R$-groups between $p$-adic inner forms of $SL_n$} 

\author[Kwangho Choiy and David Goldberg]{Kwangho Choiy and David Goldberg} 

\address{Kwangho Choiy\\
Department of Mathematics\\
Oklahoma State University\\
Stillwater, OK 74078-1058\\
U.S.A.}
\email{kwangho.choiy@okstate.edu}

\address{David Goldberg\\
Department of Mathematics \\
Purdue University\\
West Lafayette, IN 47907\\
U.S.A.}
\email{goldberg@math.purdue.edu}
\date{\today}
\begin{document}
\maketitle  

\begin{abstract} We study the Knapp-Stein $R$--groups for inner forms of the split group $SL_n(F),$ with $F$ a $p$--adic field of characteristic zero. Thus, we consider the groups $SL_m(D),$ with $D$ a central division algebra over $F$  of dimension $d^2,$ and $m=n/d.$ We use the generalized Jacquet-Langlands correspondence and results of the first named author to describe the zeros of Plancherel measures.  Combined with a study of the behavior of the stabilizer of representations by elements of the Weyl group we are able to determine the Knapp-Stein $R$--groups in terms of those for $SL_n(F).$  We show the $R$--group for the inner form embeds as a subgroup of the $R$--group for the split form, and we characterize the quotient. We are further able to show the Knapp-Stein $R$--group is isomorphic to the Arthur, or Endoscopic $R$--group as predicted by Arthur.  Finally, we give some results on multiplicities and actions of Weyl groups on $L$--packets.

\end{abstract}
\setcounter{tocdepth}{10}
\tableofcontents
\section{Introduction} \label{intro}

We embark on a study of the non-discrete tempered spectrum of non-quasi-split inner forms of groups of classical types.  Our approach is to use a generalized version of the Jacquet-Langlands correspondence to transfer information about the tempered spectrum of the split form to the inner form.  Here we study inner forms of $G=SL_n(F),$ with $F$ a $p$--adic field of characteristic zero.  Any such inner form is of the form $G'=SL_m(D),$ where $D$ is a central division algebra of degree $d^2$ over $F,$ and $m=n/d.$   
Let $\tG=GL_n(F)$ and $\tG'=GL_m(D).$  Then $\tG$ and $\tG'$ are inner forms. If $\tP$ is a parabolic subgroup of $\tG,$ then $P=\tP\cap G$ is a parabolic subgroup of $G,$ and every parabolic subgroup of $G$ arises in this way.  Similarly, if $\tP'$ is a parabolic subgroup of $\tG',$ then $P'=\tP\cap G'$ is a parabolic subgroup of $G',$ and all parabolic subgroups of $G'$ arise in this way.  Now, let $\tP'=\tM'\tN'$ be the Levi decomposition of $\tP'.$  Then $\tM'\simeq GL_{m_1}(D)\times\dots\times GL_{m_k}(D),$ for some partition  $m_1+\dots+m_k=m.$ Then there is a corresponding parabolic subgroup $\tP=\tM\tN$ of $\tG,$ with 
$\tM=GL_{m_1d}(F)\times\dots\times GL_{m_kd}(F).$  Now let $M=\tM\cap G$ and $M'=\tM'\cap G'.$  There is a generalized Jacquet-Langlands correspondence between $L$--packets on $M$ and $L$--packets on $M'.$  That is, since
$M$ and $M'$ are inner forms, and 
$$\prod_{j}SL_{m_jd}(F)\subset M\subset\prod_j GL_{m_jd}(F)$$
we have a description of the $L$--packets of $M'$ in terms of those for $M.$
Briefly, we see there is a Jacquet-Langlands correspondence from $\tM$ to $\tM',$ given by the product of the Jacquet-Langlands correspondences
from $GL_{m_id}(F)$ to $GL_{m_i}(D)$ \cite{dkv,rog83}. Let $\tsigma$ be an irreducible  discrete series representation of $\tM,$ and let 
$$\phi:W_F\longrightarrow\,^LM=\prod_j GL_{m_jd}(\CC)$$ be the Langlands parameter given by \cite{ht01,he00}.
Let $\tsigma'$ be the representation corresponding to $\tsigma$ through the Jacquet-Langlands correspondence, and note $\phi$ is also the Langlands parameter for $\tsigma'.$
Then using \cite{gk82}, we see the components of $\tsigma|_M$ form an $L$--packet of $\Pi_\phi(M),$ with Langlands parameter $\operatorname{pr}\circ\phi: W_F\rightarrow\,^LM,$
where $\operatorname{pr}$ is the projection from $GL_n(\CC)$ to $PGL_n(\CC).$  Similarly, the components of $\tsigma'|_{M'}$ form an $L$--packet, $\Pi_{\phi}(M')$ of $M',$ and we say $\Pi_\phi(M)$ and $\Pi_\phi(M')$ correspond by a generalized Jacquet-Langlands correspondence.  In fact, such a correspondence always exists between discrete series of inner forms $H$ and $H'$ if
$$\prod_{j=i}^k SL_{\ell_j}(F)\subset H\subset \prod_{j=1}^k GL_{\ell_j}(F),$$ \cite{choiy1}.
If $\sigma\in \Pi_\phi(M)$ and $\sigma'\in\Pi_\phi(M')$ then we say $\sigma$ and $\sigma'$ correspond under this generalized Jacquet-Langlands type correspondence.

Let $A$ be the split component of $M,$ and $A'$ the split component of $M'.$  We denote the reduced roots by $\Phi(P,A)=\Phi(P',A').$   Let $W_M=W_{M'}$ be the Weyl group $N_G(A)/A.$  We identify these two Weyl groups for the purpose of this exposition.
The Knapp-Stein $R$--group, $R_\sigma,$ along with a $2$--cocycle, determines the structure of the (normalized)  induced representation $i_{GM}(\sigma),$ and similarly we have the Knapp-Stein $R$--group $R_{\sigma'}$ attached to $i_{G'M'}(\sigma').$  The $R$--group, $R_\sigma$  can be realized as a quotient of two subgroups of $W_M.$  The first is the stabilizer $W(\sigma)$ of $\sigma$ in $W_M.$  The second, $W'_\sigma$ is generated by the root reflections in the zeros of the rank 1 Plancherel measures.
By the results of \cite{choiy1}, we have the transfer of the Plancherel measures (cf. Proposition \ref{pro for delta identity} and its proof).  In particular, for any $\beta\in\Phi(P,A),$ we have $\mu_\beta(\sigma)=\mu_\beta(\sigma').$  This shows $W'_\sigma=W'_{\sigma'}.$  Thus, it is left to describe $W(\sigma).$  In \cite{go94sl} the second named author showed
 \begin{equation}
 \label{stabilizer intro1}
 W(\sigma)=\left\{w\in W_M |{^w}\tsigma\simeq\tsigma \otimes \eta\text{ for some character } \eta\in\left(\tM/M\right)^D\right\},
 \end{equation}
 where the superscript $D$ indicates the Pontrjagin dual.  As in \cite{sh83,go94sl} it is straightforward to see
\begin{equation}\label{stabilizer intro2}
W(\sigma')\subset \left\{w\in W_{M'} |{^w}\tsigma'\simeq\tsigma' \otimes \eta\text{ for some character }\eta\in\left(\tM'/M'\right)^D\right\}.
\end{equation}
 The proof of equality of equation \eqref{stabilizer intro1}  in \cite{go94sl} relied on multiplicity one of restriction from $GL_n(F)$ to $SL_n(F).$  We know this multiplicity one property fails for restriction from $GL_m(D)$ to $SL_m(D),$ \cite{hs11},    It is straightforward to see the right hand sides of equations \eqref{stabilizer intro1} and \eqref{stabilizer intro2} are equal (cf. Proposition \ref{analogue of goldberg lemma} and its proof) under the identification of  $W_M$ and $W_{M'}.$  Thus, $W(\sigma')\subset W(\sigma),$ and hence $R_{\sigma'}\subset R_\sigma$ (cf.  Theorem \ref{analogue of goldberg thm2.4}).  We also give a characterization of the quotient $R_\sigma/R_{\sigma'}$ (cf. Remark \ref{meaning for W*}).

  We note the study of $R$--groups is crucial to understanding the elliptic tempered spectrum of reductive groups \cite{arthur elliptic}. Further, the isomorphism of Knapp-Stein $R$--groups and the Arthur $R$ groups play an important role in the trace formula, and in particular in the transfer of automorphic forms \cite{arthur book}.

 In Section 2 we recall the background information we need on reducibility and $R$--groups. We also discuss the structure of inner forms and specify the results we need to $SL_n(F)$ and its inner forms.  We describe the generalized Jacquet-Langlands correspondence, as described in \cite{choiy1}.  In Section 3 we describe the elliptic $L$--packets of Levi subgroups of $SL_n(F)$ and its inner forms.  In Section 4 we prove our main results on the  structure of $R$--groups for the inner forms of $SL_n(F).$  Finally, in Section 5 we discuss some results on multiplicity which arise for the inner forms of $SL_n(F).$  In an appendix also give a description of the action of characters of $M$ and $M'$ on the $L$--packets of these two groups.
 
As we finalized this manuscript we noted a preprint by K.F. Chao and W.W. Li  \cite{chaoli}, posted to the Arxiv a few days prior, which addresses the same problem.  We note our results are derived independently and based on the work of the first named author on the transfer of Plancherel measures via the generalized Jacquet-Langlands correspondence, which is a different approach than in
\cite{chaoli}, where they work mostly with the $R$--groups on the dual side. They are able to give a description of the cocyle $\eta$ and in particular give examples where $\eta$ is non-trivial.  This gives examples of induced representations with abelian $R$--groups which decompose with multiplicity greater than one.  We  also point out example 6.3.4 of \cite{chaoli} shows the containment $R_{\sigma'}\subset R_{\sigma}$ can be proper.

\section*{Acknowledgements}
The authors wish to thank Mahdi Asgari, Wee Teck Gan, Wen-Wei Li, Alan Roche, Gordan Savin, Freydoon Shahidi, Sug Woo Shin, and Shaun Stevens for communications and discussions which improved the quality of the results in this manuscript.


\section{Preliminaries} \label{pre}
In this section, we recall background materials
and review known facts.
\subsection{Notation}
Throughout this paper, $F$ denotes a $p$-adic field of characteristic $0,$ that is, a finite extension of $\QQ_p,$ with an algebraic closure $\bar{F}.$
Let $\bG$ be a connected reductive group over $F.$  We let $G=\bG(F)$ and use similar notation for other algebraic groups.
Fix a minimal $F$-parabolic subgroup $\bP_0$ of $G$ with Levi component $\bM_0$ and unipotent radical $\bN_0.$ Let $\bA_0$ be the split component of $\bM_0,$ that is, the maximal $F$-split torus in the center of $\bM_0.$ Let $\Delta$ be the set of simple roots of $\bA_0$ in $\bN_0.$ Let $\bP \subseteq \bG$ be a standard (that is, containing $\bP_0$) $F$-parabolic subgroup of $\bG.$ 
Write $\bP=\bM\bN$ with its Levi component $\bM=\bM_{\theta} \supseteq \bM_0$ generated by a subset $\theta \subseteq \Delta$ and its unipotent radical $\bN \subseteq \bN_0.$
We denote by $\delta_P$ the modulus character of $P.$ 

Let $\bA_\bM$ be the split component of $\bM.$ Denote by $X^{*}(\bM)_F$ the group of $F$-rational characters of $\bM.$ We denote by $\Phi(P, A_M)$ the reduced roots of $\bP$ with respect to $\bA_\bM.$ Denote by $W_\bM = W(\bG, \bA_\bM) := N_\bG(\bA_\bM) / Z_\bG(\bA_\bM)$ the Weyl group of $\bA_\bM$ in $\bG,$ where $N_\bG(\bA_\bM)$ and $Z_\bG(\bA_\bM)$ are respectively the normalizer and the centralizer of $\bA_\bM$ in $\bG.$ For simplicity, we write $\bA_0 = \bA_{\bM_0}.$
 
By abuse of terminology, we make no distinction between the set of isomorphism classes with the set of representatives. Let $\Irr(M)$ denote the set of isomorphism classes of irreducible admissible representations of $M=\bM(F).$ For any $\sigma \in \Irr(M),$ we write $\ii_{GM} (\sigma)$ for the normalized (twisted by $\delta_{P}^{1/2}$) induced representation.
Denote by $\Pi_{\disc}(M)$ the set of discrete series representations of $M.$

We denote by $H^i(F, \bG) := H^i(\Gal (\bar{F} / F), \bG(\bar{F}))$
the (nonabelian) Galois cohomology of $\bG.$ Denote by $W_F$ the Weil group of $F$ and $\Gamma_F := \Gal(\bar{F} / F).$ Fixing $\Gamma_F$-invariant splitting data, we define the Langlands dual ($L$-group) of $G$ as a semi-direct product $^{L}G := \widehat{G} \rtimes \Gamma_F$ (see \cite[Section 2]{bo79}).  For any topological group $H,$ we denote by $\pi_0(H)$ the group $H/H^\circ$ of connected components of $H,$ where $H^\circ$ denotes the identity component of $H.$  By $Z(H)$ we will denote the center of $H.$ 
We write $(H )^D$ for the group $\Hom(H , \CC^{\times})$ of all continuous characters. We say a character is unitary if its image is in the unit circle $S^1 \subset \CC^{\times}.$ 

For a central division algebra $D$ over $F,$ we let  $GL_{m}(D)$ denote the group of all invertible elements of $m \times m$ matrices over $D,$ and let $SL_{m}(D)$ be  the subgroup of elements in $GL_{m}(D)$ whose reduced norm is $1$ (see \cite[Sections 1.4 and 2.3]{pr94}).

For any finite set $S,$ we write $\left|{S}\right|$ for the cardinality of $S.$ For two integers $a$ and $b,$ $a \big{|} b$ means that $b$ is divisible by $a.$
\subsection{$R$-groups} \label{section for def of R}
For $\sigma\in\Irr(M)$ and $w\in W_M,$ we let ${^w}\sigma$ be the representation given by ${^w}\sigma(x)=\sigma(w^{-1}xw).$
Given $\sigma \in \Pi_{\disc}(M),$ we define
\[ W(\sigma) := \{ w \in W_M : {^w}\sigma \s \sigma \}.
\]
Set $\Delta'_\sigma = \{ \beta \in \Phi(P, A_M) : \mu_{\beta} (\sigma) = 0 \},$ where $\mu_{\beta} (\sigma)$ is the Plancherel measure attached to $\sigma$ \cite[p.1108]{go94}. Denote by $W'_{\sigma}$ the normal subgroup of $W(\sigma)$ generated by the reflections in the roots of $\Delta'_\sigma.$ 
\textit{The Knapp-Stein $R$-group} is defined by
\[
R_{\sigma} := \{ w \in W(\sigma) : w \beta > 0, \; \forall \beta \in \Delta'_\sigma \}.
\]
We write $C(\sigma):={\End}_{G}(\ii_{GM} (\sigma))$ for the algebra of $G$-endomorphisms of $\ii_{GM} (\sigma).$ We call it the commuting algebra of $\ii_{GM} (\sigma).$

\begin{thm}[Knapp-Stein \cite{ks72}; Silberger \cite{sil78, sil78cor}]  \label{thm for Knapp-Stein-Sil}
For any $\sigma \in \Pi_{\disc}(M),$ we have
\[
W(\sigma) = R(\sigma) \ltimes W'_{\sigma}. 
\]
Moreover, $C(\sigma) \s \CC[R(\sigma)]_{\eta},$ the group algebra of $R(\sigma)$ twisted by a $2$-cocycle $\eta,$ which is explicitly defined in terms of group $W(\sigma).$
\end{thm}

Let $\phi : W_F \times SL_2(\CC) \rightarrow \widehat{M}$ be an $L$-parameter.
We denote by $C_{\phi}(\widehat{M})$ the centralizer of the image of $\phi$ in $\widehat{M}$ and by $C_{\phi}(\widehat{M})^{\circ}$ its identity component. Fix a maximal torus $T_{\phi}$ in $C_{\phi}(\widehat{M})^{\circ}.$ We define
\[
W_{\phi}^{\circ} := N_{C_{\phi}(\widehat{M})^{\circ}} (T_{\phi}) /  T_{\phi},~ W_{\phi} := N_{C_{\phi}(\widehat{M})} (T_{\phi}) /  T_{\phi} ~ \text{and}~ R_{\phi}:=W_{\phi}/W_{\phi}^{\circ}.
\]
Note that $W_{\phi}$ can be identified with a subgroup of $W_M$ (see \cite[p.45]{art89ast}). Let $\Pi_{\phi}(M)$ be the $L$-packet associated to the $L$-parameter $\phi.$ For $\sigma \in \Pi_{\phi}(M),$ we set
\[
W_{\phi, \sigma}^{\circ} :=  W_{\phi}^{\circ} \cap W(\sigma),~  W_{\phi, \sigma} :=  W_{\phi} \cap W(\sigma) ~ \text{and}~ R_{\phi, \sigma}:=W_{\phi, \sigma}/W_{\phi, \sigma}^{\circ}.
\]
We call $R_{\phi, \sigma}$ \textit{the Arthur $R$-group}. 
\begin{conj}
Let $\sigma \in \Pi_{\phi}(M)$ be a discrete series representation. Then we have
\[
R_{\sigma} \s R_{\phi, \sigma}.
\]
\end{conj}
\subsection{Inner Forms} \label{inner forms}
Let $\bG$ and $\bG'$ be connected reductive groups over $F.$ We say that $\bG$ and $\bG'$ are \textit{$F$-inner forms} with respect to an $\bar{F}$-isomorphism $\varphi: \bG' \overset{\sim}{\rightarrow} \bG$ if $\varphi \circ \tau(\varphi)^{-1}$ is an inner automorphism ($g \mapsto xgx^{-1}$) defined over $\bar{F}$ for all $\tau \in \Gal (\bar{F} / F)$ (see \cite[p.851]{shin10}). If there is no confusion, we often omit the references to $F$ and $\varphi.$ Set $\bG^{\ad}:=\bG / Z(\bG).$ We note \cite[p.280]{kot97} that there is a bijection between $H^1(F, \bG^{\ad})$ and the set of isomorphism classes of $F$-inner forms of $\bG.$

Suppose $\bG$ is either $GL_n$ or $SL_n$ over $F.$ Then the set of isomorphism classes of $F$-inner forms of $\bG$ is in bijection with the subgroup $Br(F)_n$ of $n$-torsion elements in the Brauer group $Br(F)$ (see \cite[Section 2.3]{choiy1}). Hence the group $G'$ of $F$-rational points of any $F$-inner form $\bG'$ of $GL_n$ (respectively,  $SL_n$)  is of the form $GL_m(D)$ (respectively,  $SL_m(D)$) where $D$ is a central division algebra of dimension $d^{2}$ over $F$ and $n = md.$
\subsection{Structure of Levi Subgroups of $SL_n$ and its inner forms} \label{structure of levi}
Let $\tbG$ be $GL_n$ over $F.$ Let  $\widetilde{\bold P}_0$ be the minimal $F$--parabolic subgroup   of upper triangular matrices in $\tbG.$ Set the minimal $F$-Levi subgroup $\tbA_0 = \tbM_0$ to be the group of diagonal matrices, and set the unipotent radical $\tbN_0$ to be the group of unipotent upper triangular matrices. Let $\widetilde{\Delta}$ be the simple roots of $\tbA_0$ in $\tbN_0,$ so $\widetilde\Delta =\{e_i - e_{i+1} : 1 \leq i \leq n-1 \}.$ 
Let $\tbM=\tbM_{\theta}$ be an $F$-Levi subgroup of $\tbG$ for some subset $\theta \subseteq \Delta.$ Then $\tM$ is of the form $\prod_{i=1}^{k} GL_{n_i}(F)$ for some positive integer $k$ and $n_i.$ 

Let $\tbG'$ be an $F$-inner form of $\tbG,$ and let $\tbM'$ be an $F$-Levi subgroup of $\tbG'$ that is an $F$-inner form of $\tbM.$ Then $\tG'$ is of the form $GL_m(D)$ for some central division algebra $D$ of dimension $d^{2}$ over $F$ with $n = md,$ and $\tM'$ is of the form $\prod_{i=1}^{k} GL_{m_i}(D)$ with $n_i = m_i d_i.$ 

Let $\bG$ be $SL_n$ over $F.$ Let $\bP_0=\tbP_0 \cap \bG$ be our fixed minimal $F$--parabolic subgroup of $\bG.$ Set the minimal $F$-Levi subgroup $\bA_0 = \bM_0$ to be $\tbM_0 \cap \bG,$ and set the unipotent radical $\bN_0$ to be $\tbN_0 \cap \bG = \tbN_0.$ We identify  the simple roots $\Delta$ with $\widetilde{\Delta}.$
Let $\bM$ be an $F$-Levi subgroup of $\bG.$ Then we have
\begin{equation} \label{cond on M}
\prod_{i=1}^{k} SL_{n_i} \subseteq \bM \subseteq \prod_{i=1}^{k} GL_{n_i}. 
\end{equation} 

Let $\bG'$ be an $F$-inner form of $\bG,$ and Let $\bM'$ be an $F$-Levi subgroup of $\bG'$ that is an $F$-inner form of $\bM'$ Then $G'$ is of the form $SL_{m}(D),$ and $M'$ is of the form
\begin{equation} \label{cond on M' in body}
\prod_{i=1}^{k} SL_{m_i}(D) \subseteq M' \subseteq \prod_{i=1}^{k} GL_{m_i}(D).
\end{equation} 
The Weyl groups $W_M,$ $W_{\tM},$ $W_{M'}$ and $W_{\tM'}$ can be all identified and realized as a subgroup of the group $S_k$ of permutations on $k$ letters.

\begin{rem} \label{conn center}
We have the following commutative diagram of algebraic groups: 
\[
\begin{CD}
@. 1 @. 1 @. 1 @.  \\
@.      @VVV          @VVV   @VVV  @.\\
1 @>>> \bM_{\der} @>>> \bM @>>> {GL_1}^{k-1} @>>> 1 \\
@.      @|          @VVV   @VVV  @.\\
1 @>>> \tbM_{\der} @>>> \tbM @>>> {GL_1}^{k} @>>> 1 \\
@.      @VVV          @VVV   @VVV  @.\\
@. 1 @>>> \tbM / \bM @>{\s}>> GL_1 @>>> 1 \\
@.      @.          @VVV   @VVV  @.\\
@.  @. 1 @. 1 @. 
\end{CD}
\]
The maps $\bM \rightarrow {GL_1}^{k-1}$ and $\tbM \rightarrow {GL_1}^{k}$ are 
\[
(g_1, g_2, \cdots, g_{k-1}, g_) \mapsto (\det g_1, \det g_2, \cdots, \det g_{k-1})
\]
and 
\[
(g_1, g_2, \cdots, g_k) \mapsto (\det g_1, \det g_2, \cdots, \det g_k),
\] 
respectively,  with $\det$  the determinant map. Further, the map ${GL_1}^{k} \rightarrow GL_1$ is the product 
\[
(a_1, a_2,  \cdots, a_k) \mapsto a_1 \cdot a_2  \cdots a_k
\]
so that the composite $\tbM \rightarrow GL_1^k \rightarrow GL_1$ becomes the product of determinants.
We then obtain an exact sequence (the middle vertical one)
\begin{equation} \label{tilde M}
\begin{CD}
1 @>>> \bM @>>> \tbM @>>> GL_1 @>>> 1
\end{CD}
\end{equation}
which yields
\[
\begin{CD}
1 @>>> \CC^{\times} @>>> Z({\widehat{\tM}}) \s (\CC^{\times})^k @>>> Z({\widehat{M}}) @>>> 1
\end{CD}
\]
(cf. \cite[(1.8.1) p.616]{kot84}). We note that the injective map $\CC^{\times} \hookrightarrow (\CC^{\times})^k$ is a diagonal embedding. Thus the center $Z({\widehat{M}}) \s (\CC^{\times})^k / \CC^{\times}$ is connected. As another point of view, it follows from \cite[(1.8.3) p.616]{kot84} that the center $Z({\widehat{M}})$ is connected since $\bM_{\der} = \prod_{i=1}^k SL_{n_i}$ is simply connected. 
\end{rem}
\begin{rem} \label{L-groups}
We also have the following commutative diagram of $L$-groups:
\[
\begin{CD}
1 @>>> \CC^{\times} @>{\s}>> \ker @>>> 1 @. @. 
\\
@.      @VVV        @VVV   @VVV  @.\\
1 @>>> (\CC^{\times})^k @>>> \widehat{\tM} @>>> \widehat{(\tM_{\der})} @>>> 1 
\\
@.      @VVV          @VVV   @|  @.
\\
1 @>>> (\CC^{\times})^k / \CC^{\times} @>>> \widehat{M} @>>> \widehat{(M_{\der})} @>>> 1 \\
@.      @VVV          @VVV   @VVV  @.\\
@.  1 @. 1 @. 1 @. 
\end{CD}
\]
So, we have an exact sequence (the middle vertical one)
\[
\begin{CD}
1 @>>> \CC^{\times} @>>> \widehat{\tM} @>>> \widehat{M} @>>> 1.
\end{CD}
\]
Moreover, from \eqref{tilde M} we see
\[
\CC^{\times} = \widehat{GL_1} = \widehat{(\tM/M)}.
\]  
Hence, we have 
\begin{equation} \label{hat M}
\begin{CD}
1 @>>> \widehat{(\tM/M)} @>>> \widehat{\tM} @>>> \widehat{M} @>>> 1
\end{CD}
\end{equation}
is also exact.
\end{rem}
\begin{rem} \label{rem for isom between two quotients}
All arguments in Remarks \ref{conn center} and \ref{L-groups} hold for the $F$-inner form $\bM'$ of $\bM$ as well.
Further, we have a group isomorphism 
\begin{equation} \label{gp iso}
\tM/M \s F^{\times} \s \tM'/M'.
\end{equation}
Indeed, applying Galois cohomology to \eqref{tilde M}, we have
\[
1 \longrightarrow \bM(F) \longrightarrow \tbM(F) \longrightarrow F^{\times} \longrightarrow H^1(F, \bM) \longrightarrow H^1(F, \tbM) \longrightarrow 1.
\]
Note that $H^1(F, \tbM) = 1$ from \cite[Lemma 2.2]{pr94}. Also, it is well known (cf. \cite[p.270]{kot97}) that $H^1(F, \bM) \hookrightarrow H^1(F, \bG)$ (true for any connected reductive $F$-group $\bG$ and its $F$-Levi subgroup $\bM$). Since $\bG=SL_n$ is simply connected semi-simple, we have $H^1(F, \bM) = H^1(F, \bG) = 1$ due to \cite[Theorem 6.4]{pr94}. Therefore, we have $\tM/M \s F^{\times}.$ This argument is also true for $M'$ since $H^1(F, \tbM') = 1$ (\cite[Lemma 2.8]{pr94}) and $\bG'$ is simply connected semi-simple as well. Thus, we have the isomorphism \eqref{gp iso}.
\end{rem}

\subsection{Local Jacquet-Langlands correspondence} \label{local JL}
Let $\bG$ be $GL_n$ over $F$ and let $\bG'$ be an $F$-inner form of $\bG.$ We denote by $\esq(G)$ and $\esq(G')$ the sets of essentially square-integrable representations in $\Irr(G)$ and $\Irr(G'),$ respectively. We say a semisimple element $g \in G$ is \textit{regular} if its characteristic polynomial has distinct roots in  $\bar{F}$ .
We write $G^{\reg}$ for the set of regular semisimple elements in $G.$ We denote by $C^{\infty}_{c}(G)$ the Hecke algebra of locally constant functions on $G$ with compact support. Fix a Haar measure $dg$ on $G.$ For any $\rho \in \Irr(G),$ there is a unique locally constant function $\Theta_{\rho}$ on $G^{\reg}$ which is invariant under conjugation by $G$ such that
\[
\operatorname{tr} \rho (f) = \int_{G^{\reg}} \Theta_{\rho}(g)f(g)dg,
\]
for all $f \in C^{\infty}_{c}(G).$ 
Here 
\[
\rho(f)\cdot v=\int_{G}f(x)\rho(x)v\,dx.
\]
In positive characteristic, it is required that the support of the function $f \in C^{\infty}_{c}(G)$ is contained in $G^{\reg}.$
We refer the reader to \cite[p.96]{hc81} and \cite[b. p.33]{dkv} for details. The same is true for $G'.$ 

We say $g\in G^{reg}$ and $g'\in G'^{reg}$ correspond and denote this by $g\leftrightarrow g'$ if their characteristic polynomials are equal. We state the local Jacquet-Langlands correspondence (\cite[B.2.a]{dkv}, \cite[Theorem 5.8]{rog83}, and \cite[Theorem 2.2]{ba08}) as follows.
\begin{pro}  \label{proposition of local JL for essential s i}
There is a unique bijection $\mathbf{C} : \esq(G) \longrightarrow \esq(G')$ such that: for all $\sigma \in \esq(G),$ we have
\[
\Theta_{\sigma}(g) = (-1)^{n-m} \Theta_{\mathbf{C}(\sigma)}(g') 
\]
whenever $g\leftrightarrow g'.$ 
\end{pro}
\begin{rem} \label{transfer characters via JL}
For any $\sigma \in \esq(G)$ and character $\eta$ of $F^{\times},$ we have $\mathbf{C}(\sigma \otimes (\eta \circ \det)) = \mathbf{C}(\sigma) \otimes (\eta \circ \Nrd)$ due to \cite[Introduction d.4)]{dkv}. Here $\Nrd$ is the reduced norm on $G'.$ It is known that any character on $G$ (respectively, $G'$) is of the form $\eta \circ \det$ (respectively, $\eta \circ \Nrd$) for some character $\eta$ on $F^{\times}.$ 
The local Jacquet-Langlands correspondence can be generalized to the case that $G$ is a product of a general linear groups in an obvious way. 
\end{rem}
\subsection{Restriction of representations} \label{section for Tadic results}
We recall the results of Tadi{\'c} in \cite{tad92}. 
Throughout Section \ref{section for Tadic results}, $\bG$ and $\tbG$ denote connected reductive groups over $F,$ such that
\begin{equation} \label{cond on G}
\bG_{\der} = \tbG_{\der} \subseteq \bG \subseteq \tbG,
\end{equation}
where $\bG_{\der}$ and $\tbG_{\der}$ denote the derived groups of $\bG$ and $\tbG,$ respectively.
\begin{pro} (\cite[Lemma 2.1 and Proposition 2.2]{tad92}) \label{prop 2.2 tadic}
For any $\sigma \in \Irr(G),$ there exists $\tsigma \in \Irr(\tG)$ such that $\sigma \hookrightarrow \widetilde {\sigma} |_{G},$ that is, $\sigma$ is isomorphic to an irreducible constituent of the restriction $\widetilde {\sigma} |_{G}$ of $\tsigma$ to $G.$
\end{pro}
\noindent Given $\tsigma \in \Irr(G)$ as in Proposition \ref{prop 2.2 tadic}, we denote by $\Pi_{\tsigma}(G)$ the set of equivalence classes of all irreducible constituents of $\tsigma|_{G}.$ 
\begin{rem} (\cite[Proposition 2.7]{tad92}) \label{remark for sc to sc}
Any member in $\Pi_{\tsigma}(G)$ is supercuspidal, essentially square-integrable, discrete series or tempered if and only if $\widetilde{\sigma}$ is.
\end{rem}
\begin{pro} (\cite[Corollary 2.5]{tad92}) \label{pro for lifting}
Let $\tsigma_1$ and $\tsigma_2 \in \Irr(\tG)$ be given. Then the following statements are equivalent:
\begin{enumerate}[(i)]
  \item  There exists a character $\eta \in (\tG / G)^D$ such that $\tsigma_1 \s \tsigma_2 \otimes \eta;$
  \item  $\Pi_{\tsigma_1}(G) \cap \Pi_{\tsigma_2}(G) \neq \emptyset;$
  \item  $\Pi_{\tsigma_1}(G) = \Pi_{\tsigma_2}(G).$
\end{enumerate}
\end{pro}

Let $\sigma \in \Pi_{\disc}(G)$ be given. Choose $\tsigma \in \Irr(G)$ as in Proposition \ref{prop 2.2 tadic}. For any $\tilde{g} \in \tG,$ we define an action $\tilde{x}\mapsto {^{\tilde{g}}{\tsigma}}(\tilde{x}):=\tsigma(\tilde{g}^{-1} \tilde{x} \tilde{g})$ on the space of $\tsigma.$ Since $G$ is a normal subgroup of $\tG,$ the restriction of the action of $\tG$ on $\tsigma$ to $G$ induces the action $x \mapsto {^{\tilde{g}}{\sigma}}(x):=\sigma(\tilde{g}^{-1} x \tilde{g})$ on the space of $\sigma.$ 
It is clear that $^{\tilde{g}}{\tsigma} \s \tsigma$ as representations of $\tG.$ Hence, it turns out that $\tG$ acts on the set $\Pi_{\tsigma}(G)$ by conjugation. 
\begin{lm}  \label{lemma about transitive action}
The group $\tG$ acts transitively on the set $\Pi_{\tsigma}(G).$ 
\end{lm}
\begin{proof}
Let $\sigma_1$ and $\sigma_2$ be given in $\Pi_{\tsigma}(G).$ Denote by $V_i$ the space of $\sigma_i$ for $i=1, 2.$ Since $\tsigma$ is irreducible, there exists $\tilde{g} \in \tG$ such that $\tsigma(\tilde{g})V_1 = V_2.$ So, we have $\tsigma(\tilde{g}) \sigma_1 \tsigma(\tilde{g}^{-1}) \s \sigma_2.$ Hence, we have $\tilde{g} \in \tG$ such that $^{\tilde{g}}{\sigma}_1 \s \sigma_2.$
\end{proof}
\noindent We define the stabilizer of $\sigma$ in $\tG$  
\[
\tG_{\sigma}:= \{ \tilde{g} \in \tG : {^{\tilde{g}}{\sigma}} \s \sigma
\}.
\]
It is known \cite[Corollary 2.3]{tad92} that $\tG_{\sigma}$ is an open normal subgroup of $\tG$ of finite index and satisfies
\[
Z({\tG}) \cdot G \subseteq \tG_{\sigma} \subseteq \tG.
\]
Hence, we have proved the following proposition (cf. \cite[Lemma 2.1]{gk82}).
\begin{pro}  \label{simply transitive on the set}
Let $\sigma \in \Pi_{\disc}(G)$ be given. Choose $\tsigma \in \Irr(G)$ as in Proposition \ref{prop 2.2 tadic}. The quotient of $\tG/\tG_{\sigma}$ acts by conjugation on the set $\Pi_{\tsigma}(G)$ simply and  transitively. In fact, there is a bijection between $\tG/\tG_{\sigma}$ and $\Pi_{\tsigma}(G).$
\end{pro}
\noindent Note that the index $[\tG :  (Z({\tG}) \cdot G)]$ is also finite since  $\tbG$ and $\bG$ share the same derived group by the condition \eqref{cond on G} (cf. Section \ref{size of L-packet}).
\begin{rem} \label{remark for lifting}
Due to due to Propositions \ref{pro for lifting} and \ref{simply transitive on the set}, the set $\Pi_{\tsigma}(G)$ is finite and independent of the choice of $\tsigma.$
\end{rem}
\section{Elliptic tempered $L$-packets for Levi subgroups} \label{L-packets}
In this section we construct elliptic tempered $L$-packets for $F$-Levi subgroups of $SL_n$ and its $F$-inner form. Our main tools are the local Langlands correspondence for $GL_n$ in \cite{ht01, he00}, the local Jacquet-Langlands correspondence in Section \ref{local JL} and the result of Labesse in \cite{la85}. Further, we discuss the size of our $L$-packets in terms of Galois cohomology (Proposition \ref{size of L-packets}).

Throughout Section \ref{L-packets}, we continue with the notation in Section \ref{pre}. Let $\bG$ and $\bG'$ be $SL_n$ and its $F$-inner form. Let $\bM$ and $\bM'$ be $F$-Levi subgroups of $\bG$ and $\bG'$ as in \eqref{cond on M} and \eqref{cond on M' in body}, respectively. We identify $Z(\bM)$ and $Z({\bM'}).$ Note that $\bM$ is split over $F,$ so we take $^{L}M = \widehat{M} .$ Similarly, we take $^{L}\tM = \widehat{\tM} = \prod_{i=1}^{k} GL_{n_i}(\CC)$ (cf. \eqref{hat M}). Note that $Z({\widehat{M}})^{\Gamma_F} = Z({\widehat{M}})$ and $Z({\widehat{\tM}})^{\Gamma_F} = Z({\widehat{\tM}}).$ Here, the superscript ${\Gamma_F}$ means the group of ${\Gamma_F}$-invariants. We identify $\widehat{M}$ and $\widehat{M'}.$
\subsection{Construction of $L$-packets} \label{Construction of $L$-packets}
Let $\phi : W_F \times SL_2(\CC) \rightarrow \widehat{M}$ be an $L$-parameter.
We denote by $C_{\phi}(\widehat{M})$ the centralizer of the image of $\phi$ in $\widehat{M}.$ We note that $Z({\widehat{M}}) \subseteq C_{\phi}(\widehat{M}).$ We denote by $S_{\phi}(\widehat{M})$ the group $\pi_0(C_{\phi}(\widehat{M})) := C_{\phi}(\widehat{M}) / C_{\phi}(\widehat{M})^{\circ}$ of connected components.

\begin{defn}
We say that $\phi : W_F \times SL_2(\CC) \rightarrow \widehat{M}$ is \textit{elliptic} if the quotient group $C_{\phi}(\widehat{M}) / Z({\widehat{M}})$ is finite; and $\phi$ is \textit{tempered} if $\phi(W_F)$ is bounded. 
\end{defn}

\begin{rem}
The parameter $\phi$ is elliptic if and only if the image of $\phi$ in $\widehat{M}$ is not contained in any proper Levi subgroup of $\widehat{M}.$ Moreover, this is equivalent to requiring that the connected component $C_{\phi}(\widehat{M})^{\circ}$ is contained in $Z({\widehat{M}})$ (cf. \cite[(10.3)]{kot84}).
\end{rem}
The rest of Section \ref{Construction of $L$-packets} is devoted to a construction of $L$-packets on $M$ and $M'$ associated to an elliptic tempered $L$-parameter. Let $\phi : W_F \times SL_2(\CC) \rightarrow \widehat{M}$ be an elliptic tempered $L$-parameter. Recall the exact sequence \eqref{hat M}
\[
\begin{CD}
1 @>>> \widehat{(\tM/M)} \s \CC^{\times} @>>> \widehat{\tM} @>{pr}>> \widehat{M} @>>> 1.
\end{CD}
\]
By \cite[Th\'{e}or\`{e}m 8.1]{la85}, we have an (elliptic tempered) $L$-parameter
\[
\widetilde{\phi} : W_F \times SL_2(\CC) \rightarrow \widehat{\widetilde{M}}
\]
of $\tM$ such that ${pr} \circ \widetilde{\phi} = \phi$ (see \cite{weil74} and \cite{he80} for the case $M=SL_n$ and $\tM=GL_n$). 
\begin{rem} \label{choice of tilde phi}
Such a parameter $\widetilde{\phi}$ is determined up to a $1$-cocycle of $W_F$ in $\widehat{(\tM/M)}$ (see \cite[Section 7]{la85} where $\widehat{(\tM/M)}$ is a central torus $^L{S^{\circ}}$). 
\end{rem}
By the local Langlands correspondence for $GL_n$ \cite{ht01, he00}, we have a unique discrete series representation $\widetilde{\sigma} \in \Pi_{\disc}(\widetilde{M})$ associated to the $L$-parameter $\widetilde{\phi}.$ We define an $L$-packet 
\[
 \Pi_{\phi}(M) := \{\tau : \tau \hookrightarrow  \widetilde{\sigma} |_{M}  \}.
 \]
Note that, if $\widetilde{\phi}_1$ is another lift of $\phi,$ then we have $\widetilde{\phi}_1 \s \widetilde{\phi} \otimes \chi,$ for some 1-cocycle $\chi$ of $W_F$ in $\widehat{(\tM/M)},$ due to Remark \ref{choice of tilde phi}.
Hence, by the local Langlands correspondence for $GL_n$ again, $\widetilde{\phi}_1$ gives another discrete series representation $\widetilde{\sigma} \otimes \chi \in \Pi_{\disc}(\widetilde{M}).$ Here $\chi$ denotes the character on $\tM/M$ associated to the 1-cocycle $\chi$ (by abuse of notation). It is then clear $(\widetilde{\sigma} \otimes \chi)|_{M}=\widetilde{\sigma}|_{M}.$  Hence, the representation $\tsigma \otimes \chi$ associated to $\widetilde{\phi}_1$ via the local Langlands correspondence gives the same $L$-packet $\Pi_{\phi}(M)$ on $M.$

On the other hand, given $\sigma \in \Pi_{\disc}(M),$ there exists a lift $\tsigma \in \Pi_{\disc}(\tM)$ such that $\sigma \hookrightarrow \tsigma|_{M}$ (see Section \ref{section for Tadic results}). By the local Langlands correspondence for $GL_n$ \cite{ht01, he00}, we have a unique elliptic tempered $L$-parameter $\widetilde{\phi} : W_F \times SL_2(\CC) \rightarrow \widehat{\widetilde{M}}$ corresponding to $\tsigma.$ Hence, we obtain an elliptic tempered $L$-parameter $\phi : W_F \times SL_2(\CC) \rightarrow \widehat{M}$ by composing with the projection $\widehat{\widetilde{M}} \twoheadrightarrow \widehat{M}.$
If we choose another lift $\tsigma_1 \in \Pi_{\disc}(\tM)$ of $\sigma,$ then we have $\tsigma_1 \s \tsigma \otimes \chi$ for some character $\chi$ on $\tM$ such that $\chi$ is trivial on $M$ (see Section \ref{section for Tadic results}). So, the local Langlands correspondence for $GL_n$ yields the $L$-parameter $\widetilde{\phi} \otimes \chi$ corresponding to $\tsigma_1.$ But, considering the projection $\operatorname{pr}: \widehat{\widetilde{M}} \twoheadrightarrow \widehat{M},$ the composite $\operatorname{pr} \circ (\widetilde{\phi} \otimes \chi)$ has to be identical to $\phi$ 
since $\operatorname{pr}\circ\chi$ vanishes on $\widehat M.$ Hence we have verified that, given $\sigma \in \Pi_{\disc}(M),$ we have a unique elliptic tempered $L$-parameter $\phi$ corresponding to $\sigma.$ Therefore, the $L$-packets of the form $\Pi_{\phi}(M)$ exhaust all irreducible discrete series representations of $M.$

Consider $\phi$ and $\widetilde{\phi}$ as $L$-parameters of $M'$ and $\tM',$ respectively. The local Jacquet-Langlands correspondence gives a unique discrete series representation $\widetilde{\sigma}' \in \Pi_{\disc}(\widetilde{M}')$ which corresponds to the above $\widetilde{\sigma} \in \Pi_{\disc}(\widetilde{M}).$ We define an $L$-packet 
\[
 \Pi_{\phi}(M') := \{\tau' : \tau' \hookrightarrow  \widetilde{\sigma}' |_{M'}  \}
\]
(cf. \cite[Chapter 11]{hs11}). In the same way with the split case $M,$ the $L$-packets of the form $\Pi_{\phi}(M')$ exhaust all irreducible discrete series representations of $M'$ (cf. \cite[Section 4]{gk82} and \cite[Chapter 12]{hs11}). 

Through a natural embedding $\widehat{M} \hookrightarrow \widehat{G}=PGL_n(\CC),$ 
we define an $L$-packet $\Pi_{\phi}(G)$ of $G$ associated to the $L$-parameter $\phi : W_F \times SL_2(\CC) \rightarrow \widehat{G}$ as
\[
\Pi_{\phi}(G) := \{
\text{all irreducible constituents of}~ \ii_{GM}(\tau) : \tau \in \Pi_{\phi}(M)
\}.
\]
In a similar manner we define an $L$-packet $\Pi_{\phi}(G')$ of $G'.$ Recall that $\widehat{G}=\widehat{G'}.$ We let $C_{\phi}(\widehat{G})$ be the centralizer of the image of $\phi$ in $\widehat{G}.$ We denote by $S_{\phi}(\widehat{G})$ the group $\pi_0(C_{\phi}(\widehat{G})) := C_{\phi}(\widehat{G}) / C_{\phi}(\widehat{G})^{\circ}$ of connected components.
\subsection{Sizes of $L$-packets} \label{size of L-packet}
We recall from Section \ref{section for Tadic results} that, for any $\sigma \in \Pi_{\disc}(M),$ $\tM$ acts on the set $\Pi_{\tsigma}(M)$ by conjugation and the action is transitive. Note that $\Pi_{\tsigma}(M) = \Pi_{\phi}(M),$ where $\phi$ is the $L$-parameter associated to $\sigma.$ 
\begin{lm} \label{simply transitive}
Let $\Pi_{\phi}(M)$ be an $L$-packet associated to an elliptic tempered $L$-parameter $\phi : W_F \times SL_2(\CC) \rightarrow \widehat{M}.$ Then the quotient of $\tM/\tM_{\sigma}$ acts by conjugation on $\Pi_{\phi}(M)$ simply and  transitively. In fact, there is a bijection between $\tM/\tM_{\sigma}$ and $\Pi_{\phi}(M).$ The same is true for an $L$-packet $\Pi_{\phi}(M')$ of $M'.$
\end{lm}
\begin{proof}
This is a consequence of Proposition \ref{simply transitive on the set} and our construction of $L$-packets in Section \ref{Construction of $L$-packets}.
\end{proof}
 We now  consider the cardinality of the $L$-packets $\Pi_{\phi}(M)$ and $\Pi_{\phi}(M').$
\begin{pro} \label{size of L-packets}
Let $\Pi_{\phi}(M)$ be an $L$-packet associated to an an elliptic tempered $L$-parameter $\phi : W_F \times SL_2(\CC) \rightarrow \widehat{M}.$ Then we have
\[
\left| \Pi_{\phi}(M') \right| 
\; \Big{|} \; 
\left| \Pi_{\phi}(M) \right|
~ \text{and} ~ 
\left| \Pi_{\phi}(M) \right| 
\; \Big{|} \; 
\left| H^1(F, \pi_0(Z(\bM))) \right|.
\]
\end{pro}
\begin{proof}
We denote by $\bold m(\sigma)$ and $\bold m(\sigma')$ the multiplicities of $\sigma$ and $\sigma'$ in $\tsigma|_{M}$ and $\tsigma'|_{M'},$ respectively. 
We set 
\begin{align*}
Y(\tsigma) &:= \{ \eta \in (\tM/M)^D : 
   \tsigma \s \tsigma \otimes \eta \},
   \\
Y(\tsigma') &:= \{ \eta \in (\tM'/M')^D : 
   \tsigma' \s \tsigma' \otimes \eta \}.
\end{align*}
From Remarks \ref{rem for isom between two quotients} and \ref{transfer characters via JL} we see $Y(\tsigma')=Y(\tsigma).$ 
Then, from \cite[Proposition 2.4]{tad92}, we have
\[
\left| Y(\tsigma') \right| = \dim_{\CC}{\Hom}_{M'}(\tsigma', \tsigma').
\]
We note that $\dim_{\CC}{\Hom}_{M'}(\tsigma', \tsigma')$ equals the product of $\bold m(\sigma')^2$ and the number of irreducible inequivalent constituents in $\tsigma'|_{M'}$ by Schur's lemma.
Since $\Pi_{\tsigma'}(M')$ is the set of irreducible inequivalent constituents in $\tsigma'|_{M'}$ by definition, we thus have
\begin{equation} \label{equality in mul}
\left| Y(\tsigma') \right| = \left| \Pi_{\tsigma'}(M') \right| \cdot \bold m(\sigma')^2.
\end{equation}
Note  \eqref{equality in mul} holds for $M,$ as well. Since the multiplicity $\bold m(\sigma)$ of $\sigma$ in $\tsigma|_{M}$ equals $1,$ by \cite[Proposition 2.8]{tad92}, we have $\left| Y(\tsigma) \right| =\left| \Pi_{\tsigma}(M) \right|.$ Hence, we have proved the first assertion $\left| \Pi_{\phi}(M') \right| 
\; \Big{|} \; 
\left| \Pi_{\phi}(M) \right|.$

Consider a homomorphism $\lambda : Z({\bM}) \times \bM \rightarrow \tbM$ defined as $\lambda(z, m) = zm.$
Since $\bM$ and $\tbM$ have the same derived group, we get an exact sequence of algebraic groups
\begin{equation} \label{exact seq centers}
1 
\rightarrow 
Z(\bM) 
\rightarrow
Z({\tbM}) \times \bM 
\overset{\lambda}{\rightarrow}
\tbM
\rightarrow
1.
\end{equation}
We note that the injection of $Z(\bM)$ into $Z({\widetilde \bM})\times \bM$ is $z\mapsto (z,z^{-1}).$  
Applying Galois cohomology to \eqref{exact seq centers}, we have an exact sequence
\[ 
\cdots \rightarrow
Z({\tM}) \times M 
\overset{\lambda}{\rightarrow}
\tM
\rightarrow
H^{1}(F, Z(\bM))
\rightarrow H^{1}(F, Z({\tbM}) \times \bM) \rightarrow \cdots.
\]
Note that $H^{1}(F, Z({\tbM}) \times \bM) = H^{1}(F, Z({\tbM})) \times H^{1}(F, \bM) =1$ (see Remark \ref{rem for isom between two quotients}) and the cokernel of $\lambda$ is $\tM / (Z({\tM}) \cdot M).$
Consider another exact sequence 
\[
1 \rightarrow Z(\bM)^{\circ} \rightarrow Z(\bM) \rightarrow \pi_0(Z(\bM)) \rightarrow 1.
\]
Since $H^{1}(F, Z(\bM)^{\circ}) = 1$ by Hilbert's Theorem 90, we have 
\[
H^{1}(F, Z(\bM)) \hookrightarrow H^{1}(F, \pi_0(Z(\bM))).
\]
Hence, $\tM/(Z({\tM}) \cdot M)$ must be a subgroup of $H^{1}(F, \pi_0(Z(\bM))).$ 
So, we have 
\[ 
\tM / (Z({\tM}) \cdot M)
\hookrightarrow  
H^{1}(F, Z(\bM)) 
\hookrightarrow 
H^{1}(F, \pi_0(Z(\bM))). 
\]
Since $(Z({\tM}) \cdot M) \subseteq \tM_{\sigma} \subseteq \tM$ \cite[Corollary 2.3]{tad92}, we deduce the second assertion from Lemma \ref{simply transitive}. Thus, the proof is complete. 
\end{proof}
\begin{exa}
Suppose $\bG=GL_{n_1 + n_2},$ and $\bM=GL_{n_1} \times GL_{n_2}.$ Since $Z(\bM)$ is connected, any $L$-packet of $M$ is a singleton which has been proved in \cite{ht01, he00}.
\end{exa}
\begin{exa}
Suppose $F=\QQ_p,$ $p \neq 2,$ $M=G=SL_2(F)$ and $M'=G'=SL_1(D),$ where $D$ is the quaternion division algebra over $F.$ Note that $H^1(F, \pi_0(Z(\bM))) \s F^{\times} /  (F^{\times})^2,$ since $Z(\bM) = Z({\bM'}) \s \mu_2 = \pi_0(Z(\bM)).$ It turns out that $$\left| H^1(F, \pi_0(Z(\bM))) \right| = 4.$$ Hence, by Proposition \ref{size of L-packets}, the cardinality of any an elliptic tempered $L$-packet of $M$ is either $1,$ $2$ or $4.$ Also, the cardinality of an elliptic tempered $L$-packets of $M'$ can be determined using the first assertion of  Proposition \ref{size of L-packets}.
\end{exa}
\begin{cor} \label{cor about size of L-packet}
If $\pi_0(Z(\bM))=1,$ that is, $Z(\bM)$ is connected, then every an elliptic tempered $L$-packet of $M$ is a singleton.
\end{cor}
\begin{exa}
Suppose $F=\QQ_p,$ $\bG=SL_3$ and $\bM = (GL_2 \times GL_1) \cap SL_3.$ Since the coordinate ring $F[Z(\bM)] = F[x, y] / (x^2y-1)$ and $x^2y-1$ is irreducible, $Z(\bM)$ is connected. Hence, due to Corollary \ref{cor about size of L-packet}, any an elliptic tempered $L$-packet of $M$ is singleton. This argument is clear for the obvious reason that $\bM \s GL_2.$
\end{exa}
\begin{rem}
Let $\phi : W_F \times SL_2(\CC) \rightarrow \widehat{G}$ be a tempered $L$-parameter. Let $\widehat{M}$ be a minimal Levi subgroup in the sense that $\widehat{M}$ contains the image of $\phi$ (see \cite[Section 3.4]{bo79} for the definition of Levi subgroup of $^{L}M$). Then $\phi$ becomes an elliptic tempered parameter of $M.$ Choose a member $\tau \in \Pi_{\phi}(M)$ and recall the Knapp-Stein $R$-group $R_{\tau}$ for $\tau.$ Then, from Lemma \ref{simply transitive} we have 
\[
\left| \Pi_{\phi}(G)  \right|  
\; \Big{|} \; 
\big( \left| R_{\tau}  \right| \cdot \left| \Pi_{\phi}(M)  \right| \big)
\; \Big{|} \; 
\big( \left| R_{\tau}  \right| \cdot \left| H^1(F, \pi_0(Z(\bM)))  \right| \big).
\]
The same is true for $\bG'.$ In fact, for the split case $\bG$, we have an equality
\[
\left| \Pi_{\phi}(G)  \right|  
= 
\left| R_{\tau}  \right| \cdot \left| \Pi_{\phi}(M)  \right|
\]
from Theorem \ref{goldberg thm2.4} and Proposition \ref{Gelbart-Knapp} in Section \ref{R-group for SL}.
\end{rem}
\begin{rem}
All statements in Section \ref{L-packets} admit an obvious generalization to the case of any connected reductive group $\bM$ over $F$ such that 
$
\bM_{\der} = \tbM_{\der} \subseteq \bM \subseteq \tbM,
$
where $\tbM = \prod_{i=1}^{k}GL_{n_i}$.
\end{rem}
\section{$R$-groups for $SL_n$ and its Inner Forms} \label{R-groups for SL and SL(D)}

In this section we first review the results of the second named author \cite{go94sl} and Gelbart-Knapp \cite{gk82} and prove that Knapp-Stein, Arthur and Endoscopic $R$-groups are all identical for $SL_n$ (Theorem \ref{conc}). Second, we discuss the Knapp-Stein $R$-group for an $F$-inner form of $SL_n$ and establish its connection with $R$-groups for $SL_n.$ Throughout Section \ref{R-groups for SL and SL(D)}, we continue with the notation in Sections \ref{pre} and \ref{L-packets}.
\subsection{$R$-groups for $SL_n$ Revisited} \label{R-group for SL}
Let $\tbG = GL_n$ and $\bG = SL_n$ be over $F.$
Let $\bM$ be an $F$-Levi subgroup of $\bG,$ and let $\tbM$ be an $F$-Levi subgroup of $\tbG$ such that $\tbM = \bM \cap \bG.$
Let $\sigma \in \Pi_{\disc}(M)$ be given. Choose $\tsigma \in \Irr(\tM)$ such that $\sigma \hookrightarrow \tsigma|_{M}.$ We identify Weyl groups $W_M$ and $W_{\tM}$ as a subgroup of the group $S_k$ of permutations on $k$ letters. Let  $\phi : W_F \times SL_2(\CC) \rightarrow \widehat{M}$ be an elliptic tempered $L$-parameter associated to $\sigma.$ Let $\Pi_{\phi}(M)$ be an $L$-packet of $M$ associated to $\phi.$
We define the following groups:
\begin{align*}
\bar{L}(\tsigma) &:= \{ \eta \in (\tM/M)^D : \; 
  ^w \tsigma \s \tsigma \otimes \eta ~ \text{for some} ~ w \in W_M \},
\\
X(\tsigma) &:= \{ \eta \in (\tM/M)^D : 
   \tsigma \s \tsigma \otimes \eta \},
\\
X({\ii}_{\tG \tM}(\tsigma)) &:= \{ \eta \in (\tG/G)^D : {\ii}_{\tG\tM}(\tsigma) \s {\ii}_{\tG \tM}(\tsigma) \otimes \eta \}.
\end{align*}
Since any character of $GL_n(F)$ is of the form $\eta \circ \det$ for some character of $F^{\times},$ 
we often make no distinct between $\eta$ and $\eta \circ \det.$ Further, since $\tbM$ is of the form $\prod_{i=1}^{k} GL_{n_i},$ we simply write $\eta$ for $\prod_{i=1}^{k} (\eta_i \circ \det_i) \in (\tM/M)^D,$ where $\eta_i$ denotes a character of $F^{\times}$ and $\det_i$ denotes the determinant of $GL_{n_i}(F)$ for each $i.$

\begin{lm} (Goldberg, \cite[Lemma 2.3]{go94sl}) \label{lemma by goldberg}
Let $\sigma \in \Pi_{\disc}(M)$ be given. For any lift $\tsigma \in \Pi_{\disc}(\tM)$ such that $\sigma \hookrightarrow \tsigma |_{M},$ we have
\[
W(\sigma) = \{ w \in W_M : \; 
  ^w \tsigma \s \tsigma \otimes \eta ~ \text{for some} ~ \eta \in (\tM/M)^D
\}.
\]
\end{lm}

\begin{thm} (Goldberg, \cite[Theorem 2.4]{go94sl}) \label{goldberg thm2.4} 
The Knapp-Stein $R$-group $R_{\sigma}$ is isomorphic to 
$\bar{L}(\tsigma) / X(\tsigma).$
\end{thm}
\noindent Note that ${\ii}_{\tG \tM}(\tsigma)$ is always irreducible since $\tsigma$ is a discrete series representation \cite[Theorem 4.2]{bz77}.
\begin{pro} (Gelbart-Knapp, \cite[Theorem 4.3]{gk82}) \label{Gelbart-Knapp}
There are group isomorphisms
\[
X(\tsigma) \s S_{\phi}(\widehat{M}) ~ \text{and} ~ X({\ii}_{\tG \tM}(\tsigma)) \s S_{\phi}(\widehat{G}).
\]
Consequently, both $S_{\phi}(\widehat{M})$ and $S_{\phi}(\widehat{G})$ are finite abelian groups. Further, $S_{\phi}(\widehat{M})^D$ and $S_{\phi}(\widehat{G})^D$ have  canonical simply transitive group actions on $\Pi_{\phi}(M)$ and $\Pi_{\phi}(G),$ respectively. 
\end{pro}
\begin{rem}
Since $X(\tsigma)$ is a subgroup of $X({\ii}_{\tG \tM}(\tsigma)),$ we notice that there is an embedding $\iota : S_{\phi}(\widehat{M}) \hookrightarrow S_{\phi}(\widehat{G}).$ We abuse notation to denote the quotient  $S_{\phi}(\widehat{G}) / \iota (S_{\phi}(\widehat{M}))$ by $S_{\phi}(\widehat{G}) / S_{\phi}(\widehat{M}).$ We call the quotient $S_{\phi}(\widehat{G}) / S_{\phi}(\widehat{M})$ \textit{the Endoscopic $R$-group}. We refer the reader to \cite[(7.1)]{art89ast} where the Endoscopic $R$-group is denoted by $R_{\phi}.$ 
\end{rem}
We have the following proposition.
\begin{pro} \label{pro 2 for thm 1} 
We have the identity  
\[
\bar{L}(\tsigma) = X({\ii}_{\tG \tM}(\tsigma))
\]
as subgroups of the group $(F^{\times})^D$ of characters on $F^{\times}.$
\end{pro}
\begin{proof}
Let $\eta \in (\tG/G)^D$ be given. 
We first note that ${\ii}_{\tG \tM}(\tsigma) \otimes \eta  \s {\ii}_{\tG \tM}(\tsigma \otimes \eta ).$
From the Langlands classification \cite[Theorem 1.2.5 (b)]{ku94}, we have 
\[
{\ii}_{\tG \tM}(\tsigma) \s {\ii}_{\tG \tM}(\tsigma \otimes \eta )
\]
if and only if
\[
^w \tsigma \s \tsigma \otimes \eta  ~ \text{for some} ~ w \in W_M.
\]
Now, the proposition follows from the definition of $\bar{L}(\tsigma).$
\end{proof}

\begin{thm} \label{thm1} 
The Knapp-Stein $R$-group $R_{\sigma}$ is isomorphic to 
$S_{\phi}(\widehat{G}) / S_{\phi}(\widehat{M}).$
\end{thm}
\begin{proof}
This is a consequence of Theorem \ref{goldberg thm2.4} and Propositions \ref{Gelbart-Knapp} and  \ref{pro 2 for thm 1} (cf. \cite[Proposition 9.1]{sh83} for another proof).
\end{proof}
In what follows, we give a connection between the Knapp-Stein $R$-group $R_{\sigma}$ and the Arthur $R$-group $R_{\phi, \sigma}$ for the case when $\sigma$ is in $\Pi_{\disc}(M)$ and $\bM$ is an $F$-Levi subgroup of $\bG=SL_n.$
\begin{lm} \label{equality for Arthur R-group SL}
We have the equalities
\[
W_{\phi} = W_{\phi, \sigma} ~ \text{and} ~ W_{\phi}^{\circ} = W_{\phi, \sigma}^{\circ}.
\]
\end{lm}
\begin{proof}
The former implies the latter. So, it is enough to show that $W_{\phi} \subset W_{\phi, \sigma}.$ Let $w \in W_{\phi}$ be given. From \cite[Lemma 2.3]{baj04}, we note that the element $w$ lies in $W_{\widehat{M}}$ satisfying $^w \phi \s \phi,$ i.e., $^w \phi$ is conjugate to $\phi$ in $\widehat{M}$ (since $w \in C_{\phi}(\widehat{M})$). It follows that $^w{\tilde{\phi}} \s \tilde{\phi} \otimes \eta$ for some $\eta \in (F^{\times})^D.$ We identify $W_M$ and $W_{\widehat{M}}.$ From the local Langlands correspondence for $GL_n$ and the $W_M$-action on the $L$-parameter $\tilde{\phi} = \tilde{\phi}_1 \oplus \cdots \oplus \tilde{\phi}_k,$ we have
\[
^w{\tsigma} \s \; \tsigma \otimes \eta
\]
for some $\eta \in (\tM/M)^D.$ Now,  from Lemma \ref{lemma by goldberg} we see $w$ must be in $W(\sigma).$ By the definition $W_{\phi, \sigma} = W_{\phi} \cap W(\sigma),$ which completes the proof.
\end{proof}
\noindent Hence, from \cite[(7.1) \& p.44]{art89ast},  we get 
\begin{equation} \label{an equ}
R_{\phi} = R_{\phi, \sigma} \s S_{\phi}(\widehat{G}) / S_{\phi}(\widehat{M}).
\end{equation}
Thus, due to Theorems \ref{goldberg thm2.4} and \ref{thm1}, we have proved the following.
\begin{thm} \label{conc}
Let $\bM$ be an $F$-Levi subgroup of $\bG=SL_n.$ Let $\Pi_{\phi}(M)$ be an $L$-packet on $M$ associated with an elliptic tempered $L$-parameter $\phi : W_F \times SL_2(\CC) \rightarrow \widehat{M}.$ For any $\sigma \in \Pi_{\phi}(M),$ we have
\[
R_{\phi} = R_{\phi, \sigma} \s S_{\phi}(\widehat{G}) / S_{\phi}(\widehat{M}) 
\s \bar{L}(\tsigma) / X(\tsigma) \s R_{\sigma}.
\]
\end{thm}

\subsection{$R$-groups for $F$-inner Forms of $SL_n$} \label{R-group for SL(r,D)}
Let $\bM'$ be an $F$-Levi subgroup of an inner form $\bG'$ of $\bG=SL_n.$  
Let $\sigma' \in \Pi_{\disc}(M')$ be given. Choose $\tsigma' \in \Pi_{\disc}(\tM')$ such that $\sigma' \hookrightarrow \tsigma'|_{M'}.$ 

As in Section \ref{structure of levi}, we identify $W_M,$ $W_{M'}$ and $W_{\tM'}.$ We also identify $\tM/M$ and $\tM'/M'$ (see Remark \ref{rem for isom between two quotients}). In the same manner as in Section \ref{R-group for SL}, we define
\begin{align*}
\bar{L}(\tsigma') &:= \{ \eta' \in (\tM'/M')^D : \; 
  ^w \tsigma' \s \tsigma' \otimes \eta' ~ \text{for some} ~ w \in W_{M'} \},
\\
X(\tsigma') &:= \{ \eta' \in (\tM'/M')^D : 
   \tsigma' \s \tsigma' \otimes \eta' \},
\\
X({\ii}_{\tG' \tM'}(\tsigma')) &:= \{ \eta' \in (\tG'/G')^D : {\ii}_{\tG'\tM'}(\tsigma') \s {\ii}_{\tG' \tM'}(\tsigma') \otimes \eta' \}.
\end{align*}
Since any character of $GL_m(D)$ is of the form $\eta \circ \Nrd$ for some character of $F^{\times}$ (cf. \cite[Proposition 53.5]{bh06}), we often make no distinct between $\eta$ and $\eta \circ \Nrd.$ Further, since $\tM'$ is of the form $\Pi_{i=1}^{r} GL_{m_i}(D),$ we simply write $\eta$ for $\Pi_{i=1}^{r} (\eta_i \circ \Nrd_i)  \in (\tM'/M')^D,$ where $\eta_i$ denotes a character of $F^{\times}$ and $\Nrd_i$ denotes the reduced norm of $GL_{m_i}(D)$ for each $i.$

\begin{pro} \label{analogue of goldberg lemma}
Let $\sigma' \in \Pi_{\disc}(M')$ be given. Choose $\tsigma' \in \Pi(\tM')$ such that $\sigma' \hookrightarrow \tsigma'|_{M'}.$ Then, we have
\[
W(\sigma') \subseteq \{ w \in W_{M'} : \; 
  ^w \tsigma' \s \tsigma' \otimes \eta' ~ \text{for some} ~ \eta' \in (\tM'/M')^D
\} = W(\sigma).
\]
\end{pro}
\begin{proof}
Let $w \in W(\sigma')$ be given. Then it turns out that $^w \tsigma'$ and $\tsigma'$ share an irreducible constituent ${^w\sigma} \s \sigma.$ So, Proposition \ref{pro for lifting} implies that $^w \tsigma \s \tsigma \otimes \eta'$ for some character $\eta' \in (\tM'/M')^D.$ This proves the first inclusion. 

For the second inclusion, we note that $^w \tsigma'$ and $\tsigma' \otimes \eta'$ correspond respectively to $^w \tsigma$ and $\tsigma \otimes \eta'$ by the local Jacquet-Langlands correspondence (see Section \ref{local JL}). So we have from Remarks \ref{rem for isom between two quotients} and \ref{transfer characters via JL} 
\[
^w \tsigma' \s \tsigma' \otimes \eta' ~ \text{for some} ~ \eta' \in (\tM'/M')^D
\]
if and only if
\[
^w \tsigma \s \tsigma \otimes \eta ~ \text{for some} ~ \eta \in (\tM/M)^D.
\]
Here we identify $\eta$ and $\eta'$ via the isomorphism $\tM'/M' \s  \tM/M$ in Remark \ref{rem for isom between two quotients}. Applying Lemma \ref{lemma by goldberg}, the proof is complete.
\end{proof}

From now on, we let $\phi : W_F \times SL_2(\CC) \rightarrow \widehat{M}=\widehat{M'}$ be an elliptic tempered $L$-parameter. 
Let $\Pi_{\phi}(M)$ and $\Pi_{\phi}(M')$ be $L$-packets of $M$ and $M'$ associated to $\phi,$ respectively.

\begin{rem}
Let $\sigma \in \Pi_{\phi}(M)$ and $\sigma' \in \Pi_{\phi}(M')$ be given.
Lemma \ref{lemma by goldberg} and Proposition \ref{analogue of goldberg lemma} provide the following diagram.
\[
\xymatrix{
{^w}\tsigma \s \tsigma \otimes \eta 
\ar@2{<->}[d] \ar@2{<->}[r]^{\Tiny{LJL}}
&
{^w}\tsigma' \s \tsigma' \otimes \eta 
\ar@{-->}[d] \\
{^w}\sigma \s \sigma \ar@{<-}[r]^{--->} 
&{^w}\sigma' \s \sigma' \ar@<1ex>[u]  ,
}
\]
where $w$ lies in $W_M=W_{M'}=W_{\tM}=W_{\tM'}$ and $\eta$ is a character on $\tM/M \s \tM'/M'$ (see Remark \ref{rem for isom between two quotients}).
\end{rem}
\begin{pro} \label{pro for delta identity}
For any $\sigma \in \Pi_{\phi}(M)$ and $\sigma' \in \Pi_{\phi}(M'),$ we have
\[
W'_{\sigma} = W'_{\sigma'}
\]
as subgroups of the group $S_r$ of permutations on $r$ letters.
\end{pro}
\begin{proof}
We note that the sets $\Phi(P, A_M)$ and $\Phi(P', A_{M'})$ of reduced roots are identical from \cite[p.85]{go94sl}. Let $\beta \in \Phi(P, A_M) = \Phi(P', A_{M'})$ be given. Since the Plancherel measures are invariant on $\Pi_{\phi}(M')$ due to \cite[Corollary 7.2]{sh89}, we have from \cite[Theorem 6.3]{choiy2}

\begin{equation} \label{equ pm identity}
\mu_\beta(\sigma) = \mu_{\beta}(\sigma').
\end{equation}
In particular, $\mu_\beta(\sigma)=0$ if and only if $\mu_{\beta}(\sigma')=0.$  From the definitions of $W'_\sigma$ and $W'_{\sigma'}$ (see Section \ref{section for def of R}), the lemma follows.
\end{proof}
\begin{rem} \label{rem for pm identity}
Equation \eqref{equ pm identity} can be also deduced from \cite[Theorem 5.7]{choiy2} since our discrete series representations in $\Pi_{\phi}(M)$ and $\Pi_{\phi}(M')$ satisfy the character identity
\[
\sum_{\sigma \in \Pi_{\phi}(M)} \Theta_{\sigma}(\gamma) = (-1)^{n-m} \sum_{\sigma' \in \Pi_{\phi}(M')} \bold m(\sigma') \Theta_{\sigma'}(\gamma')
\]
for any $\gamma$ and $\gamma'$ have matching conjugacy classes. Here $\bold m(\sigma')$ is the multiplicity of $\sigma'$ in $\tsigma'|_{M'}$ and $\Theta_{\natural}$ is the character function (Harish-Chandra character) of $\natural.$ This character identity comes from the restriction of the character identity in the local Jacquet-Langlands correspondence between $\tM$ and $\tM'$ (Proposition \ref{proposition of local JL for essential s i}). We also notice that the multiplicity $\bold{m(\sigma)} = 1$ \cite[Theorem 1.2]{tad92}. Further, we refer the reader to \cite[(2.5) on p.88]{ac89} and \cite[Corollary 7.2]{sh89} for another approach.
\end{rem}

Now we state a relation between $R$-groups for $G=SL_n(F)$ and $G'=SL_m(D)$ as follows.
\begin{thm} \label{analogue of goldberg thm2.4}
Let $\sigma \in \Pi_{\phi}(M)$ and $\sigma' \in \Pi_{\phi}(M')$ be given. For any liftings $\tsigma \in \Pi_{\disc}(\tM)$ and $\tsigma' \in \Pi_{\disc}(\tM')$ such that $\sigma \hookrightarrow \tsigma |_{M'}$ and $\sigma' \hookrightarrow \tsigma' |_{M'},$ we have 
\[
R_{\sigma'} \hookrightarrow \bar{L}(\tsigma') / X(\tsigma') \s R_\sigma \s R_{\phi, \sigma} =  R_{\phi}.
\]
\end{thm}
\begin{proof}
From Propositions \ref{analogue of goldberg lemma} and \ref{pro for delta identity}, it follows that $R_{\sigma'} \subseteq  R_{\sigma}.$ 
Since $^w \tsigma'$ and $\tsigma' \otimes \eta$ correspond respectively to $^w \tsigma$ and $\tsigma \otimes \eta$ by the local Jacquet-Langlands correspondence (see Section \ref{local JL}), we have
\[
\bar{L}(\tsigma) = \bar{L}(\tsigma') ~ \text{and} ~ X(\tsigma) = X(\tsigma').
\]
Thus, the isomorphism $R_{\sigma} \simeq \bar{L}(\tsigma') / X(\tsigma')$ follows from Theorem \ref{goldberg thm2.4}. Hence, by Theorem \ref{conc}, we have the claim.
\end{proof}
\begin{cor}
Let $\sigma \in \Pi_{\phi}(M)$ and $\sigma' \in \Pi_{\phi}(M')$ be given. If $\ii_{GM}(\sigma)$ is irreducible, then $\ii_{G'M'}(\sigma')$ is irreducible.
\end{cor}
The following theorem asserts that Knapp-Stein $R$-groups for $G'$ are invariant on $L$-packets (cf. \cite[Corollary 2.5]{go94sl} for $G$). 
\begin{thm} \label{invariant theorem}
For any $\sigma'_1,$ $\sigma'_2$ $\in \Pi_{\phi}(M'),$ we have
\[
R_{\sigma'_1} \s R_{\sigma'_2}.
\]
\end{thm}
\begin{proof}
By fixing $\sigma \in \Pi_{\phi}(M),$ Lemma \ref{pro for delta identity} shows  $W'_{\sigma'_1} = W'_{\sigma'_2}.$ It then suffices to show that $W(\sigma'_1) \s W(\sigma'_2).$
From Lemma \ref{simply transitive}, choose $x \in \widetilde M'$ such that ${^x\sigma_1'} \s \sigma_2'.$ We define a map
\[
w' \mapsto xw'x^{-1}
\]
from $W(\sigma_1')$ to $W(\sigma_2').$ Since $M'$ is a normal subgroup of $\tM',$ the element $xw'x^{-1}$ must be in $N_{\tG'}(\bM')\cap G'=N_{G'}(M'),$ and $W(\sigma_2').$ Note that this map is an injective homomorphism. In the same manner, we obtain another injective homomorphism
\[
w' \mapsto x^{-1}w'x
\]
from $W(\sigma_2')$ to $W(\sigma_1').$ It is clear that one map is the inverse of the other. Thus, the proof is complete. 
\end{proof}
\begin{cor}
Let $\sigma'_1,$ $\sigma'_2$ be given $\in \Pi_{\phi}(M').$ Then  $\ii_{G'M'}(\sigma'_1)$ is irreducible if and only if $\ii_{G'M'}(\sigma'_2)$ is irreducible.
\end{cor}

Now we describe the difference between $R_\sigma$ and $R_{\sigma'}$ (cf. Appendix \ref{behavior of cha act} for another interpretation). We define a finite quotient
\begin{equation} \label{W*}
W^*(\sigma') := W(\sigma) / W(\sigma')
\end{equation}
from Proposition \ref{analogue of goldberg lemma}.
So, due to Proposition \ref{pro for delta identity} we have
\begin{equation} \label{diff in terms of W*}
1 \rightarrow R_{\sigma'} \rightarrow R_{\sigma} \rightarrow W^*(\sigma') \rightarrow 1,
\end{equation}
which implies $W^*(\sigma') \s R_{\sigma}/R_{\sigma'}.$
Given any $\sigma'_1,$ $\sigma'_2$ $\in \Pi_{\phi}(M'),$ we have from Theorems \ref{goldberg thm2.4} and \ref{invariant theorem}
\begin{equation}  \label{isomorphism of W*}
W^*(\sigma_1') \s W^*(\sigma_2').
\end{equation}
\begin{rem}  \label{meaning for W*}
We note that $W^*(\sigma')$ can be realized as the set
$\{ {^w\sigma'} : w \in W(\sigma) \}.$
Further, given $\sigma_1'$ and $\sigma_2'$ in $W^*(\sigma'),$ we have $JH(\sigma_1')=JH(\sigma_2').$ Here $JH(\sigma_i')$ is the set of all irreducible constituents in $\ii_{G'M'}(\sigma_i')$ for $i=1, 2.$ 
\end{rem}

\begin{rem}  We note that $W^*(\sigma')$ can be non-trivial, in which case $R_{\sigma'}\subsetneq R_\sigma,$ as is discussed in example 6.3.4 of \cite{chaoli}.
\end{rem}

\section{Multiplicities on $F$-inner forms of $SL_n$ } \label{multi}
In this section we discuss several possible 
multiplicities which occur in the restriction and the parabolic induction of an $F$-inner form $\bG'$ of $\bG=SL_n.$ We express these in terms of  cardinalities of $L$-packets the difference between $R$-groups of $SL_n$ and its $F$-inner form.
We continue with the notation in Section \ref{R-group for SL(r,D)}. 

Throughout Section \ref{multi}, we let  $\phi : W_F \times SL_2(\CC) \rightarrow \widehat{M} = \widehat{M'}$ be an elliptic tempered $L$-parameter. Let $\Pi_{\phi}(M)$ and $\Pi_{\phi}(M')$ be $L$-packets of $M$ and $M',$ respectively, associated to $\phi.$ Consider $\phi$ as an $L$-parameter of $G$ and $G'$ through a natural embedding $\widehat{M} \hookrightarrow \widehat{G} =\widehat{G'}.$   We denote by $\Pi_{\phi}(G)$ and $\Pi_{\phi}(G')$ $L$-packets of $G$ and $G',$ respectively, associated to $\phi$ (see Section \ref{Construction of $L$-packets}).

Let $\sigma \in \Pi_{\phi}(M)$ and $\sigma' \in \Pi_{\phi}(M')$ be given. Choose $\tsigma \in \Pi_{\disc}(\tM)$ and $\tsigma' \in \Pi_{\disc}(\tM')$ such that $\sigma \hookrightarrow \tsigma |_{M'}$ and $\sigma' \hookrightarrow \tsigma' |_{M'}.$ We consider the following isomorphism 
\begin{equation} \label{iso bw res and ind}
(\ii_{\tG' \tM'}(\tsigma'))|_{G'} \s \ii_{G' M'} (\tsigma'|_{M'})
\end{equation}
as $G'$-modules, since the restriction and the parabolic induction are compatible (cf. \cite[Proposition 4.1]{baj04}). Fix an irreducible constituent $\pi'$ of $\ii_{G' M'} (\sigma').$ We use the following notation.
\begin{itemize}
  \item  $\bold m(\pi')$ denotes the multiplicity of $\pi'$ in $\ii_{G' M'} (\sigma').$
  
  \item  $\bold m_{M',G'}((\ii_{\tG' \tM'}(\tsigma'))|_{G'})$ denotes the multiplicity of $\pi'$ in the restriction $(\ii_{\tG' \tM'}(\tsigma'))|_{G'}.$ 

  \item  $\bold m_{M',G'}(\tsigma'|_{M'})$ denotes the multiplicity of $\pi'$ in  $\ii_{G' M'} (\tsigma'|_{M'}).$
  
\end{itemize}
In what follows, we present these multiplicities in terms of $|W^*(\sigma')|$, 
$|\Pi_{\phi}(M)|,$ $|\Pi_{\phi}(M')|,$ $|\Pi_{\phi}(G)|$ and $|\Pi_{\phi}(G')|.$
\begin{lm} \label{multi lemma I}
We have 
\[
\bold m_{M',G'}((\ii_{\tG' \tM'}(\tsigma'))|_{G'})
= 
\sqrt{\frac{|\Pi_{\phi}(G)|}{|\Pi_{\phi}(G')|}}.
\]
\end{lm}
\begin{proof}
Consider the commuting algebra 
\[
{\End}_{G'}(\ii_{\tG' \tM'}(\tsigma')) = {\Hom}_{G'}(\ii_{\tG' \tM'}(\tsigma'), \ii_{\tG' \tM'}(\tsigma')).
\]
Applying \cite[Proposition 2.4]{tad92}, we have 
\begin{align*}
\dim_{\CC} {\End}_{G'}(\ii_{\tG' \tM'}(\tsigma')) 
&=
| \{   \eta \in (\tM/M)^D : {\ii}_{\tG \tM}(\tsigma) \s {\ii}_{\tG \tM}(\tsigma) \otimes \eta     \}| 
\\
&=
|X(\ii_{\tG' \tM'}(\tsigma'))|.
\end{align*}
On the other hand, since the restriction $(\ii_{\tG' \tM'}(\tsigma'))|_{G'}$ is completely reducible, \cite[Lemma 2.1]{tad92}, we have a direct sum
\[
(\ii_{\tG' \tM'}(\tsigma'))|_{G'} \s  \bigoplus_{ \{\tau'\} }  \bold m \; \tau'.
\]
Here $\tau'$ runs through all irreducible inequivalent constituents in $(\ii_{\tG' \tM'}(\tsigma'))|_{G'},$ and $\bold m $ is the common multiplicity of irreducible constituents in $(\ii_{\tG' \tM'}(\tsigma'))|_{G'}.$
We note that ${\Hom}_{G'}(\tau'_1, \tau'_2) = 0$ unless $\tau'_1 \s \tau'_2,$ in which case ${\Hom}_{G'}(\tau'_1, \tau'_2) \s \CC$ 
by Schur's lemma. Hence, we get
\[
{\End}_{G'}(\bigoplus_{ \{\tau'\} } \bold m \; \tau') 
\s \bigoplus_{ \{\tau'\} } {\End}_{G'}(\tau')^{\bold m^2} \s (\CC)^{\bold m^2}
\]
(cf. \cite[Lemma 2.1(d)]{gk82}).
We note that the set of all irreducible inequivalent constituent in $(\ii_{\tG' \tM'}(\tsigma'))|_{G'}$ equals the $L$-packet $\Pi_{\phi}(G')$ (see $\S$\ref{Construction of $L$-packets}). 
Replacing the common multiplicity $\bold m$ by the multiplicity $\bold m_{M',G'}((\ii_{\tG' \tM'}(\tsigma'))|_{G'})$ of $\pi',$ we thus have 
\[
\dim_{\CC}{\End}_{G'}(\ii_{\tG' \tM'}(\tsigma')) = [\bold m_{M',G'}((\ii_{\tG' \tM'}(\tsigma'))|_{G'})]^2 \cdot |\Pi_{\phi}(G')|.
\]
Note that $X(\ii_{\tG' \tM'}(\tsigma'))$ equals $X(\ii_{\tG \tM}(\tsigma)),$ which is in bijection with $\Pi_{\phi}(G)$ (Proposition \ref{Gelbart-Knapp}). Therefore, we have 
\begin{equation} \label{left equality}
\bold m_{M',G'}((\ii_{\tG' \tM'}(\tsigma'))|_{G'}) = \sqrt{\frac{|\Pi_{\phi}(G)|}{|\Pi_{\phi}(G')|}}.
\end{equation}
This proves the lemma.
\end{proof} 

\begin{lm} \label{multi lemma II} 
We have
\[
\bold m_{M',G'}(\tsigma'|_{M'})
=
\sqrt{\frac{|\Pi_{\phi}(M)|}{|\Pi_{\phi}(M')|}} \cdot |W^*(\sigma')| \cdot \bold m(\pi').
\]
\end{lm}
\begin{proof}
Fix $\sigma'$ in the restriction $\tsigma'|_{M'}$ such that $\pi'$ is an irreducible constituent of $\ii_{G'M'}(\sigma').$ We note that
\[
{\dim}_{\CC}{\Hom}_{G'}(\pi',i_{G'M'}(\tsigma'))
=m_1^2m_2m_3^2,
\]
where 
\begin{itemize}
 \item $m_1$ denotes the common multiplicity of each constituent  in $\tsigma'|_{M'}$, 
 \item $m_2$ denotes the number of inequivalent components $\sigma'$ in the restriction $\tsigma'|_{M'}$ such that $\pi'$ is an irreducible constituent of $\ii_{G'M'}(\sigma')$, and
 \item $m_3$ denotes the multiplicity of $\pi'$ in $\ii_{G' M'} (\sigma').$
\end{itemize}
Thus, the multiplicity $\bold m_{M',G'}(\tsigma'|_{M'})$ equals $m_1^2m_2m_3^2.$ In the same manner as equality \eqref{left equality}, we have 
${m_1= \sqrt{\frac{|\Pi_{\phi}(M)|}{|\Pi_{\phi}(M')|}}.}$  Further, from Remark \ref{meaning for W*}, we have   $m_2= |W^*(\sigma')|.$ We also have $m_3= \bold m(\pi')$ by definition. Therefore, the proof is complete.
\end{proof}

\begin{pro} \label{compare multi}
We have
\begin{equation} \label{equality of multis}
\sqrt{\frac{|\Pi_{\phi}(G)|}{|\Pi_{\phi}(G')|}} = \sqrt{\frac{|\Pi_{\phi}(M)|}{|\Pi_{\phi}(M')|}} \cdot |W^*(\sigma')| \cdot \bold m(\pi').
\end{equation}
\end{pro}
\begin{proof}
From the isomorphism \eqref{iso bw res and ind}, we have
\[
\bold m_{M',G'}((\ii_{\tG' \tM'}(\tsigma'))|_{G'})=\bold m_{M',G'}(\tsigma'|_{M'}).
\]
The proposition is thus a consequence of Lemmas \ref {multi lemma I} and \ref{multi lemma II}. 
\end{proof}
\begin{rem}
Let $\pi_1'$ and $\pi_2'$ be irreducible constituents of $\ii_{G' M'}(\sigma'_1)$ and $\ii_{G' M'}(\sigma'_2)$ for some $\sigma_1'$ and $\sigma_2',$ respectively, in $\Pi_{\phi}(M').$ Due to \eqref{isomorphism of W*}, it turns out that all the factors except $\bold m(\pi')$ in equality \eqref{equality of multis} do not depend on the choice of $\pi'$ . Hence, we have 
\[
\bold m(\pi'_1) = \bold m(\pi'_2).
\]
Further, both ratios $\displaystyle{\frac{|\Pi_{\phi}(G)|}{|\Pi_{\phi}(M)|}}$ and $\displaystyle{\frac{|\Pi_{\phi}(G')|}{|\Pi_{\phi}(M')|}}$ are always square integers.  
\end{rem}

\appendix
\section{Actions of Characters on $L$-packets} \label{behavior of cha act}
Throughout Appendix \ref{behavior of cha act}, we use the notation in Sections \ref{pre} and \ref{L-packets}. In this appendix, we describe the difference between actions of characters of $M$ and $M'$ on their $L$-packets and interpret the difference in $R$-groups for $SL_n$ and its $F$-inner forms in terms of characters (see \eqref{diff in terms of W*} for another interpretation).
We identify $\tM/M$ and $\tM'/M'$ (see Remark \ref{rem for isom between two quotients}).

Throughout Section \ref{R-group for SL(r,D)}, we let $\phi : W_F \times SL_2(\CC) \rightarrow \widehat{M}=\widehat{M'}$ be an elliptic tempered $L$-parameter. Let $\Pi_{\phi}(M')$ be an $L$-packet of $M'$ associated to $\phi.$ Given $\sigma' \in \Pi_{\phi}(M'),$ we define a set 
\[
X(\sigma'):=\{ \chi : M' \rightarrow S^1 ~|~ \sigma' \otimes \chi \s \sigma' \},
\] 
which consists of all unitary characters on $M'$ stabilizing $\sigma'.$ Here $S^1$ denotes the unit circle in $\CC^{\times}.$ It is easy to check that $X(\sigma')$ is an abelian group.
We define a set
\begin{equation} \label{X_phi}
X_{M'}(\phi) :=
\{
\chi : M' \rightarrow S^1 ~|~ \sigma' \otimes \chi \in \Pi_{\phi}(M') ~ \text{for some} ~ \sigma' \in \Pi_{\phi}(M')  
\},
\end{equation}
which consists of all unitary characters on $M'$ stabilizing $\Pi_{\phi}(M').$
For any $\sigma' \in \Pi_{\phi}(M'),$ it is clear that $X_{M'}(\phi) \supseteq X(\sigma').$
\begin{lm} \label{lemma abelian}
The set $X_{M'}(\phi)$ is a finite abelian group.
\end{lm}
\begin{proof}
First, we claim that, given any character $\chi : M'\rightarrow S^1,$
\begin{equation} \label{claim}
\sigma'_0 \otimes \chi \in \Pi_{\phi} ~ \text{for some} ~ \sigma'_0 \in \Pi_{\phi}(M')~\Longleftrightarrow~ \sigma' \otimes \chi \in \Pi_{\phi} ~ \text{for all} ~ \sigma' \in \Pi_{\phi}(M').
\end{equation}
Indeed, we choose $\tsigma'_0 \in \Pi_{\disc}(\tM')$ such that $\sigma'_0 \hookrightarrow \tsigma'_0|_{M'}.$ Let $\tilde{\chi} : \tM' \rightarrow S^1$ be a character whose restriction to $M'$ is identical with $\chi.$ From our construction of $L$-packets for $M',$ the set of inequivalent irreducible constituents in $\tsigma'_0 |_{M'}$ equals 
$\Pi_{\phi}(M').$ Since $
\sigma'_0 \otimes \chi$ is an irreducible constituent of both $\tsigma'_0|_{M'}
$ and $(\tsigma'_0 \otimes \tilde\chi)|_{M'},
$
we have
\begin{equation} \label{some iso}
\tsigma'_0 |_{M'} \s (\tsigma'_0 \otimes \tilde\chi)|_{M'} \s (\tsigma'_0 |_{M'}) \otimes \chi
\end{equation}
as representations of $M'.$ Thus, it follows that $\sigma' \otimes \chi$ lies in $\Pi_{\phi}$ for any $\sigma' \in \Pi_{\phi}(M').$ 

Next, we show that $X_{M'}(\phi)$ is an abelian group. Let $\chi_1$ and $\chi_2$ be in $X_{M'}(\phi).$ It suffices to show that $\chi_2 \chi^{-1}_1 \in X_{M'}(\phi)$ since $X_{M'}(\phi)$ is a subset of the abelian group of unitary characters on $M'.$ From the definition of $X_{M'}(\phi)$ we have $\sigma'_1$ and $\sigma'_2$ in $\Pi_{\phi}(M')$ such that both $\sigma'_1 \otimes \chi_1$ and $\sigma'_2 \otimes \chi_2$ lie in $\Pi_{\phi}(M').$  We set $\sigma'_* := \sigma'_1 \otimes \chi_1.$ Since $\sigma'_* \otimes \chi_1^{-1} = \sigma'_1  \in \Pi_{\phi}(M'),$ it follows from the claim \eqref{claim} that
\[
\sigma'_2 \otimes \chi_2 \chi_1^{-1} \in \Pi_{\phi}(M').
\]
Hence, we have $\chi_2 \chi^{-1}_1 \in X_{M'}(\phi).$

Finally, it remains to show that $X_{M'}(\phi)$ is finite. Let $\chi \in X_{M'}(\phi)$ be given such that $\sigma' \otimes \chi \in \Pi_{\phi}(M').$ Note that all members in $\Pi_{\phi}(M')$ have the same central character which is the restriction of the central character of $\tsigma'$ to the center $Z({M'})$ of $M'.$ Here $\tsigma'$ is a lifting of $\sigma'$ such that $\sigma' \hookrightarrow \tsigma'|_{M'}.$  So, the character $\chi$ on $M'$ is trivial on $Z({M'}),$ which implies $\chi$ is trivial on $Z({M'}) \cdot M'_{\der} .$ Applying Galois cohomology (cf. the proof of Proposition \ref{size of L-packets}), we have the exact sequence
\[
1 
\rightarrow 
Z({M'_{\der}}) 
\rightarrow
Z({M'}) \times M'_{\der}
\rightarrow
M'
\rightarrow
H^1(F, Z({\bM'_{\der}})).
\]
Since $H^1(F, Z({\bM'_{\der}}))$ is finite, we notice that $Z({M'}) \cdot M'_{\der} $ has finite index in $M'.$ It thus follows that $X_{M'}(\phi)$ is finite.
\end{proof}
\begin{rem} \label{rem for singleton}
We note that, if $\Pi_{\phi}(M')$ is a singleton, it then follows that $X(\sigma')= X_{M'}(\phi).$
\end{rem}
\begin{lm} \label{lm for cru thm}
For any $\sigma'_1, \sigma'_2 \in \Pi_{\phi}(M'),$ we have 
\[
X(\sigma'_1) = X(\sigma'_2).
\]
In other words, for any character $\chi$ on $M',$ $\sigma'_1 \s \sigma'_1 \otimes \chi$ if and only if $\sigma'_2 \s \sigma'_2 \otimes \chi.$
\end{lm}
\begin{proof}
There exists an element $g \in \tM'$ such that $\sigma'_2 \s {^g}\sigma'_1$ by Lemma \ref{lemma about transitive action} and our construction of $\Pi_{\phi}(M')$ in Section \ref{Construction of $L$-packets}. Since $M'_{\der} \subseteq M' \subseteq \tM',$ we notice that any character $\chi$ on $M'$ is the restriction of a character $\widetilde{\chi}$ on $\tM'.$ 
So, we have ${^g}\chi(m')= \widetilde{\chi}(g^{-1}m'g)=\widetilde{\chi}(m')=\chi(m')$ for any $m' \in M'.$  Thus we get 
\begin{align*}
\chi \in X(\sigma'_2) 
& ~\Leftrightarrow~ 
 \sigma'_2 \s \sigma'_2 \otimes \chi  \\
& ~\Leftrightarrow~ {^g}\sigma'_1 \s ({^g}\sigma'_1) \otimes \chi ~\Leftrightarrow~
{^g}\sigma'_1 \s {^g}(\sigma'_1 \otimes \chi) ~\Leftrightarrow~
\sigma'_1 \s \sigma'_1 \otimes \chi
~\Leftrightarrow~
\chi \in X(\sigma'_1).
\end{align*}
\end{proof}
Let $\Pi_{\phi}(M)$ be an $L$-packet of $M$ associated to $\phi.$ Given $\sigma \in \Pi_{\phi}(M),$ we define two groups $X(\sigma)$ and $X_{M}(\phi)$ in the same manner as those for $M'$
\begin{align*}
X(\sigma) &:= \{ \chi : M \rightarrow S^1 ~|~ \sigma \otimes \chi \s \sigma
\},& \\
X_{M}(\phi) &:=
\{
\chi : M \rightarrow S^1 ~|~ \sigma \otimes \chi \in \Pi_{\phi}(M) ~ \text{for some} ~ \sigma \in \Pi_{\phi}(M)  
\}.&
\end{align*}
Then we have 
\begin{equation} \label{for m}
X(\sigma) = X_{M}(\phi).
\end{equation}
Indeed, if $\sigma,\sigma \otimes \chi\in\Pi_{\phi}(M)$ then, by  \cite[Theorem 1.2]{tad92}, their restrictions to $M_{\der}$ are equivalent, and hence, by multiplicity one of restriction from $\tM,$ to $M,$ they must be isomorphic. 
We remark that this is also known in the case for $L$-packets for $GSp(4)$ \cite[Proposition 2.2]{gtsp10}.

We notice the absence of the multiplicity one  of restriction from $\tM'$ to $M'$ (cf. \cite[p.215]{art06}). We establish the following proposition which is a different phenomenon from the split form.
\begin{pro} \label{a pro}
Let $\sigma \in \Pi_{\phi}(M)$ and $\sigma' \in \Pi_{\phi}(M')$ be given. Then we have 
\[
X(\sigma') \subseteq X_{M'}(\phi) \overset{1-1}{\longleftrightarrow} X_{M}(\phi) = X(\sigma).
\]
\end{pro}
\begin{proof}
We recall any character $\tilde{\chi}$ on $\tM = \prod_{i=1}^k GL_{n_i}(F)$ is of the form 
$
\prod_{i=1}^k \tilde{\chi}_i \circ \det_i,
$
where each $\tilde{\chi}_i$ denotes a character on $F^{\times}$ and $\det_i$ denotes the determinant on $GL_{n_i}(F).$ Likewise, any character $\tilde{\chi}'$ on $\tM' = \prod_{i=1}^k GL_{m_i}(D)$ is of the form 
$
\prod_{i=1}^k \tilde{\chi}_i \circ \Nrd_i,
$
where each $\tilde{\chi}_i$ denotes a character on $F^{\times}$ and each $\Nrd_i$ denotes the reduced norm on $GL_{m_i}(D).$ Here $D$ denotes a central simple division algebra of dimension $d$ over $F$ with $n=dm,$ where $\sum_{i=1}^k n_i = n$ and $\sum_{i=1}^k m_i = m.$ Under the local Jacquet-Langlands correspondence between $\Pi_{\disc}(M)$ and $\Pi_{\disc}(M'),$  if $\tsigma$ corresponds to $\tsigma',$ then  $\tsigma \otimes \tilde\chi$ corresponds to $\tsigma' \otimes \tilde\chi',$ since each $\tilde{\chi}_i \circ \det_i$ corresponds to $\tilde{\chi}_i \circ \Nrd_i$ (see Remark \ref{transfer characters via JL}). 

Let $\eta' \in X(\sigma')$ be given. We denote by $\tilde{\eta}' : \tM' \rightarrow S^1$ a character which restricts to $\eta'$ on $M'.$  Then, using above arguments, we have
\[ \eta' \in X(\sigma') 
~\Leftrightarrow~ \sigma' \s \sigma' \otimes \eta' ~\Leftrightarrow~  \tsigma' \s \tsigma' \otimes \tilde\eta'
~\Leftrightarrow~ \tsigma \s \tsigma \otimes \tilde\eta
~\Leftrightarrow~
\sigma \s \sigma \otimes \eta
~\Rightarrow~ \eta \in X(\sigma),
\]
where $\eta$ is a character on $M$ such that $\tilde\eta|_{M} = \eta.$
\end{proof}
\begin{rem}
We note that the lifting $\tilde{\chi}$ of $\chi$ is uniquely determined up to a character on $\tM/M$ (\cite[Corollary 2.5]{tad92}) and the local Jacquet-Langlands correspondence is uniquely characterized by the character relation (see Proposition \ref{proposition of local JL for essential s i}). Hence, the bijection in Proposition \ref{a pro} becomes an isomorphism and unique. 
\end{rem}
\begin{rem}
Let $\Psi_{nr}(M)$ be the group of all unramified characters on $M.$ Given $\sigma \in \Pi_{\phi}(M),$ we set $\Stab_{\Psi_{nr}(M)}(\sigma):=\{\psi \in \Psi_{nr}(M) ~|~ \sigma \otimes \psi \s \sigma \}.$ Given $\sigma' \in \Pi_{\phi}(M'),$ we define $\Psi_{nr}(M')$ and $\Stab_{\Psi_{nr}(M')}(\sigma')$ for the inner form $M'$ in the same manner. Note that $\Psi_{nr}(M) \s \Psi_{nr}(M')$ (cf. \cite[Section 6]{roc09}). Hence, Proposition \ref{a pro} implies that 
\[
{\Stab}_{\Psi_{nr}(M')}(\sigma') \subseteq {\Stab}_{\Psi_{nr}(M)}(\sigma)
\]
for any $\sigma \in \Pi_{\phi}(M)$ and $\sigma' \in \Pi_{\phi}(M').$  
\end{rem}

\begin{pro} \label{chi}
Let $\sigma' \in \Pi_{\phi}(M')$ be given. Choose $\tsigma' \in
\Pi_{\disc}(\tM')$ such that $\sigma' \hookrightarrow \tsigma'|_{M'}.$ Let
$w$ be an element in $W_{\tM'} = W_{M'}$ such that $^w\tsigma' \s \tsigma' \otimes
\eta$ for some $\eta \in (\tM'/M')^D.$ Then there exists a unitary
character $\chi' \in M'^D$ such that
\[
^w \sigma' \s \sigma' \otimes \chi'.
\]
\end{pro}

\begin{proof}
It is enough to show that there exist an irreducible representation $\rho'$ of $M'_{\der}$ such that ${^w}\rho'$ is a component of both $\sigma'|_{M'_{\der}}$ and $(^w\sigma')|_{M'_{\der}}.$
We shall follow the idea on \cite[Lemma 2.6]{gol06} (cf. \cite[Lemma 2.3]{go94sl}). 

We note that $M'_{\der}$ is of the form
\[
SL_{m_1}(D) \times  SL_{m_2}(D) \times \cdots \times SL_{m_k}(D).
\]
Write $\tsigma' = \pi'_1 \otimes \pi'_2 \otimes \cdots \otimes \pi'_k.$ 
Let $\rho'$ be an irreducible constituent in $\sigma'|_{M'_{\der}}.$ Then $\rho'$ is isomorphic to $\rho' \s \rho'_1 \otimes \rho'_2 \otimes \cdots \otimes \rho'_k$ for some irreducible constituent $\rho'_i$ in $\pi_i'|_{SL_{m_i}(D)}.$ Suppose $w=w_1w_2 \cdots w_t$ is the disjoint cycle decomposition of $w$ by regarding $W_{\tM'}$ as a subgroup of the group $S_k$ of permutations on $k$ letters. Without loss of generality, we assume that $w_1=(1 \, 2 \cdots j).$   Since ${^w}\tsigma' \s \tsigma' \otimes \eta,$ we then have $\pi'_{i+1} \s \pi'_i \otimes \eta \s \pi'_1 \otimes \eta^i,$ for $i=1,2, \cdots, j-1,$ and $\pi'_1 \s \pi'_1 \otimes \eta^j.$  (Here, by abuse of notation, $\eta$ is regarded as a character on $GL_{m_1}(D).$) Thus, for each $1 \leq i \leq j,$ the representation $\rho'_i$ is an irreducible constituent in $\pi_1|_{SL_{m_1}(D)}.$ By Lemma \ref{lemma about transitive action}, for each $1 \leq i \leq j-1,$ there is an $a_i \in D^{\times},$ so that $\rho'_{i+1}={^{\delta(a_i)}}{\rho'}_i,$ where
\[
\delta(a) =  \left( \begin{array}{ c c }
     a &  \\
      & I_{m_1-1}
  \end{array} \right) .
\]
Let $a_j=(a_1 \, a_2 \cdots a_{j-1})^{-1}.$ Then $^{\delta(a_j)}{\rho'}_j=\rho_1.$ Set 
$$b_1 = \diag(\delta(a_1), \delta(a_2), \cdots, \delta(a_j), 1,1,\cdots,1).$$
 Then we have $\det\circ \Nrd (b_1) = 1,$ and 
 $${^{b_1}}\rho'= \rho'_2  \otimes \cdots \otimes \rho'_j \otimes \rho'_1 \otimes\rho_{j+1}\otimes\cdots\otimes\rho_k= {^{w_1}}\rho'.$$
  Similarly, for $i=2, 3, \cdots, t,$ we can find such a $b_i$ such that $\det\circ \Nrd (b_i)=1$ and $\displaystyle{^{b_i}\rho'={^{w_i}}\rho'.}$  Setting $b=\diag(b_1, \cdots, b_t, 1, \cdots, 1),$ we have $b \in M'$ and $^{b}\rho' \s {^w}\rho'.$ Therefore, since $^{b}\sigma' \s \sigma',$ the irreducible representation $^{w}\rho'$ of $M'_{\der}$ must be in both restrictions $\sigma'|_{M'_{\der}}$ and $(^{w}\sigma')|_{M'_{\der}}.$   This completes the proof.
\end{proof}
We set
\[
\bar{W} := \{ w \in W_{M'} : {^w}\tsigma' \s \tsigma' \otimes
\eta ~~\text{for some}~~\eta \in (\tM'/M')^D \}.
\]
Then Proposition \ref{chi} allows us to define a map from $\bar{W}$ to
$X_{M'}(\phi)/X(\sigma')$ by sending $w \mapsto \chi = \chi_w.$ Therefore,
given any $\sigma\in\Pi_{\phi}(M)$ and $\sigma'\in\Pi_{\phi}(M'),$ we have
from Proposition \ref{analogue of goldberg lemma} and the definition
of $W^*(\sigma')$ in \eqref{W*}
\begin{equation} \label{last}
R_{\sigma}/R_{\sigma'} \s W^*(\sigma') = W(\sigma)/W(\sigma') \s
\bar{W}/W(\sigma') \hookrightarrow X_{M'}(\phi)/X(\sigma').
\end{equation}
Here we note that the isomorphism $\bar{W} \s W(\sigma)$ comes from Lemma \ref{lemma by
goldberg} and the local Jacquet-Langlands correspondence (cf. proof of
Proposition \ref{analogue of goldberg lemma}).

\begin{rem}
Let $\phi : W_F \times SL_2(\CC) \rightarrow \widehat{M}= \widehat{M'}$ be
an elliptic tempered $L$-parameter. Assume that the $L$-packet
$\Pi_{\phi}(M')$ is a singleton. Then, by Remark \ref{rem for singleton} and \eqref{last}, we have $R_{\sigma} =
R_{\sigma'}.$
\end{rem}

\end{document}